\documentclass{article}

\usepackage{amsmath,amssymb}
\usepackage{amsthm}
\usepackage{enumitem}
\usepackage{multicol}
\usepackage{hyperref}
\usepackage{array}
\usepackage{url}

\usepackage{color}

\usepackage{rotating}
\usepackage{tikz,tkz-graph}
\tikzset{
every node/.style={circle, draw, inner sep=2pt},
every picture/.style={thick}
}
\usetikzlibrary{matrix,arrows,calc}
\usetikzlibrary{decorations.pathreplacing,decorations.markings}

\newtheorem{theorem}{Theorem}
\newtheorem{lemma}[theorem]{Lemma}

\newtheorem{proposition}[theorem]{Proposition}
\newtheorem{corollary}[theorem]{Corollary}

\theoremstyle{definition}
\newtheorem{definition}[theorem]{Definition}
\newtheorem{observation}[theorem]{Observation}
\newtheorem{remark}[theorem]{Remark}
\newtheorem{example}[theorem]{Example}
\newtheorem{question}[theorem]{Question}

\def \mr {\operatorname{mr}}
\def \rank {\operatorname{rank}}
\def \coker {\operatorname{coker}}

\newcommand{\diag}{\operatorname{diag}}
\newcommand{\SNF}{\operatorname{SNF}}

\newcommand{\lrangle}[1]{\left\langle#1\right\rangle}

\newcommand{\mz}{\operatorname{mz}}

\begin{document}

\title{The structure of sandpile groups of outerplanar graphs}
\author{Carlos A. Alfaro$^a$ and
Ralihe R. Villagr\'an$^b$
\\ \\
{\small $^a$ Banco de M\'exico} \\
{\small Mexico City, Mexico}\\
{\small {\tt alfaromontufar@gmail.com,carlos.alfaro@banxico.org.mx}} \\
{\small $^b$Departamento de Matem\'aticas}\\ {\small Centro de Investigaci\'on y de Estudios Avanzados del IPN}\\
{\small Apartado Postal 14-740, 07000 Mexico City, Mexico} \\
{\small {\tt
rvillagran@math.cinvestav.mx}}\\
}
\date{}

\maketitle

\begin{abstract}
We compute the sandpile groups of families of planar graphs having a common weak dual by evaluating the indeterminates of the critical ideals of the weak dual at the lengths of the cycles bounding the interior faces.
This method allow us to determine the algebraic structure of the sandpile groups of outerplanar graphs, and can be used to compute the sandpile groups of many other planar graph families.
Finally, we compute the identity element for the sandpile groups of the dual graphs of many outerplane graphs.
\end{abstract}


\section{Introduction}

The dynamics of the {\it Abelian sandpile model}, which was firstly studied by Bak, Tang and Wiesenfeld in \cite{btw}, is developed on a connected graph $G$ with a special vertex $q$, called {\it sink}.
We denote by $\mathbb{N}$ the set of non-negative integers.
In the sandpile model, a configuration on $(G,q)$ is a vector ${ c}\in \mathbb{N}^{V}$, in which entry  ${ c}_v$ is associated with the number of {\it grains of sand} or {\it chips} placed on vertex $v$.
Two configurations $c$ and $d$ are {\it equal} if $c_v=d_v$ for each non-sink vertex.
The sink vertex is used to collect the sand getting out the system.
A non-sink vertex $v$ is called \textit{stable} if ${ c}_v$  is less than its {\it degree} $d_G(v)$, and {\it unstable}, otherwise.
Thus, a configuration is called \textit{stable} if every non-sink vertex is stable.
The {\it toppling rule} in the dynamics of the model consists in selecting an unstable non-sink vertex $u$ and moving $d_G(u)$ grains of sand from $u$ to its neighbors, in which each neighbor $v$ receives $m_{(u,v)}$ grains of sand, where $m_{(u,v)}$ denote the number of edges between $u$ to $v$.
Note toppling vertex $v_i$ in configuration ${c}$ corresponds to the subtraction the $i$-{\it th} row of the Laplacian matrix to ${c}$. 
Recall the Laplacian matrix $L(G)$ of a graph $G$ is defined such that the $(u,v)$-entry of $L(G)$ is defined as 
\[
L(G)_{u,v}=
\begin{cases}
\deg_G(u) & \text{if } u=v,\\
-m(u,v) & \text{otherwise.}
\end{cases}
\]
Starting with any unstable configuration and toppling unstable vertices repeatedly, we will always obtain \cite[Theorem 2.2.2]{Klivans} a stable and unique configuration after a finite sequence of topplings.
The stable configuration obtained from the configuration ${ c}$ will be denoted by $s({ c})$.
The sum of two configurations ${ c}$ and ${ d}$ is performed entry by entry.
A configuration ${ c}$ is \textit{recurrent} if there exists a non-zero configuration ${ d}$ such that ${ c}=s({ d}+{ c})$.
Let ${ c}\oplus { d}:=s({ c}+{ d})$.
Recurrent configurations play a central role in the dynamics of the Abelian sandpile model since recurrent configurations together with the $\oplus$ operation form an Abelian group known as \textit{sandpile group} \cite[Chapter 4]{Klivans}.
In the following $K(G)$ denote the sandpile group of $G$.
One of the interesting features of the sandpile group of connected graphs is that the order $|K(G)|$ is equal to the number $\tau(G)$ of spanning trees of the graph $G$.

The sandpile group has been studied under different names, for example: {\it chip-firing game} \cite{biggs,merino}, {\it critical group} \cite{biggs,CY}, {\it group of components} \cite{lorenzini2008}, {\it Jacobian group} \cite{bhn,biggs}, {\it Laplacian unimodular equivalence} \cite{merris}, {\it Picard group} \cite{bhn,biggs}, or {\it sandpile group} \cite{alfaval0,CR}.
We recommend the reader the book \cite{Klivans} which is an excellent reference on the theory of chip-firing game and its relations with other combinatorial objects like rotor-routing, hyperplane arrangements, parking functions and dominoes. 
In particular, the properties of the sandpile configurations are explained in detail there.
On the other hand, the Abelian sandpile model was the first example of a {\it self-organized critical system}, which attempts to explain the occurrence of power laws in many natural phenomena ranging on different fields like geophysics, optimization, economics and neuroscience.
A nice exposure to self-organized-critically is provided in the book~\cite{ubi}.

Two matrices $M$ and $N$ are said to be {\it equivalent} if there exist $P,Q\in GL_n(\mathbb{Z})$ such that $N=PMQ$, and denoted by $N\sim M$.
Given a square integer matrix $M$, the Smith normal form (SNF) of $M$ is the unique equivalent diagonal matrix $\diag(d_1,d_2,\dots,d_n )$ whose non-zero entries are non-negative and satisfy $d_i$ divides $d_{i+1}$. 
The diagonal elements of the SNF are known as {\it invariant factors}.
In \cite{stanley}, Stanley surveys the influence of the SNF in combinatorics.
In our context the SNF is relevant since the sandpile group is isomorphic to the torsion part of the cokernel of the Laplacian matrix of $G$ \cite[Chapter 4]{Klivans}, and the SNF of a matrix is a standard technique to determine the structure of cokernel.
This is because if $N\sim M$, then $\coker(M)=\mathbb{Z}^n/{\rm Im} M\cong\mathbb{Z}^n/{\rm Im} N=\coker(N)$.
In particular, the fundamental theorem of finitely generated Abelian groups states
$\coker(M)\cong \mathbb{Z}_{d_1} \oplus \mathbb{Z}_{d_2} \oplus \cdots \oplus\mathbb{Z}_{d_{r}} \oplus \mathbb{Z}^{n-r}$, where $r$ is the rank of $M$. 
The minimal number of generators of the torsion part of the cokernel of $M$ equals the number of positive invariant factors of $\SNF(M)$. 
Let $f_1(G)$ and $\phi(G)$ denote the number of invariant factors of $L(G)$ equal to 1 and the minimal number of generators of $K(G)$, respectively.
If $G$ is a graph with $n$ vertices and $c$ connected components, then $n-c=f_1(G)+\phi(G)$.

It is important to note that the algebraic structure of the sandpile group does not depend on the sink vertex, meanwhile the combinatorial structure depicted by the recurrent configurations of $G$ do depends on the sink vertex.

At the beginning, it was found \cite{lorenzini2008,wagner} that many graphs have cyclic group from which was conjectured that almost all graphs have cyclic sandpile group.
However, it was found in \cite{wood} that the probability that the sandpile group of a random graph is cyclic is asymptotically at most $0.7935212$. 
Still, it was proved \cite{CY} that for any given connected simple graph, there is an homeomorphic graph with cyclic sandpile group.
Recall, we say that two graphs $G_1$ and $G_2$ are in the same {\it homeomorphism class} if there exists a graph $G$ that is a subdivision of both $G_1$ and $G_2$.

The following lemma is convenient in many situations to compute the invariant factors of a matrix $M$.
\begin{lemma}
For $k\in[\rank(M)]$, let $\Delta_k(M)$ be the $\gcd$ of the $k$-minors of matrix $M$, and $\Delta_0(M)=1$.
Then the $k$-{\it th} invariant factor $d_k(M)$ of $M$ equals
\[\frac{\Delta_k(M)}{\Delta_{k-1}(M)}.
\]
\end{lemma}

This relation inspired H. Corrales and C. Valencia to introduce in \cite{corrval} the critical ideals of a graph, which are determinantal ideals generalizing the sandpile group and their varieties generalize the spectrum of the graph.
Let $A(G)$ be the adjacency matrix of the graph $G$ with $n$ vertices. Let $A_X(G)=\diag(x_1,\dots,x_n)-A(G)$, where the indeterminates of $X=(x_1,\dots,x_n)$ are associated with the vertices of $G$.
For $k\in[n]$, the $k$-{\it th critical ideal} $I_k(G)$ of $G$ is the ideal in $\mathbb{Z}[X]$ generated by the $k$-minors of the matrix $A_X(G)$. 
Note the evaluation of the $k$-{\it th} critical ideal of $G$ at $X=\deg(G)$ will be an ideal in $\mathbb{Z}$ generated by $\Delta_k(L(G))$.
We will show a new application of the critical ideals for computing the sandpile group of planar graph.

When the graph is connected, it is convenient to compute the cokernel of a reduced Laplacian matrix since it is full rank.
The {\it reduced Laplacian matrix} $L_k(G)$ for a connected graph $G$ is the $(n-1)\times(n-1)$ matrix obtained by deleting the row and column $k$ from $L(G)$.
There are $n$ different reduced Laplacian matrices and $K(G)\cong \coker(L_k(G))$ and $|K(G)|=\det(L_k(G))=\tau(G)$ for any $k\in[n]:=\{1,\dots,n\}$, see details in \cite{biggs}.

We will use $G^*$ to denote the {\it dual} of a plane graph $G$, and the {\it weak dual}, denoted by $G_*$, is constructed the same way as the dual graph, but without placing the vertex associated with the outer face.
It is known \cite{Berman,CR,vince} that the sandpile group of a planar graph is isomorphic to the sandpile group of its dual.
Since the dual of any plane graph is connected \cite{BM}, then $K(G)\cong\coker(L_k(G^*))$ and $\tau(G)=\det(L_k(G^*))$.

In \cite{Phifer}, C. Phifer gave a nice interpretation of this relation by introducing the {\it cycle-intersection matrix} of a plane graph as follows.
Given a plane graph $G$ with $s$ interior faces $F_1,\dots,F_s$, let $c(F_i)$ denote the length of the cycle which bounds interior face $F_i$. 
We define the
{\it cycle-intersection matrix}, $C(G) = (c_{ij})$ to be a symmetric matrix of size $s\times s$, where $c_{ii}=c(F_i)$, and $c_{ij}$ is the
negative of the number of common edges in the cycles bounding the interior faces $F_i$ and $F_j$, for
$i\neq j$.
This matrix differs from the {\it fundamental circuits intersection matrix} used in \cite{CY}.
Note that $C(G)$ is the reduced Laplacian of $G^*$ where the column and row associated with the outer face are removed from $L(G^*)$.
Therefore we have the following. 
\begin{lemma}\label{lemma:lemmacycleintersection}
Let $G$ be a plane graph.
Then $K(G)\cong\coker(C(G))$ and $\tau(G)=\det(C(G))$.
\end{lemma}

Recently, the structure of the sandpile group of some subfamilies of the outerplanar graphs were established, see for example \cite{BG,cm,Kerpkiy}.
Also, the Tutte polynomial and the number of spanning tress of an infinite families of outerplanar, small-world and self-similar graphs were obtained in \cite{CMLZ,lfh}.
Despite this, the algebraic structure of the sandpile groups of the outerplanar graphs have been largely unknown.

In Section~\ref{sec:critgi}, we explore the relation obtained in Lemma~\ref{lemma:lemmacycleintersection} under the lenses of the critical ideals of graphs.
Then, we give a methodology to compute the algebraic structure of the sandpile groups of the plane graph family $\mathcal{F}$ that have a common weak dual. 
This method consists in evaluating the indeterminates of the critical ideals of the weak dual at the lengths of the cycles bounding the interior faces of the plane graph in $\mathcal{F}$.
In Section~\ref{sec:outerplanar}, we use this method and the property that the weak dual of outerplane graphs are trees, which was suggested by Chen and Mohar in \cite{cm}, to compute the sandpile groups of outerplanar graphs.
This result rely on previous results obtained by Corrales and Valencia in \cite{corrval1}.
Finally, in Section~\ref{sec:identity}, we compute the identity configuration for the sandpile groups of the dual graphs of many outerplane graphs.

\section{Sandpile groups of planar graphs}\label{sec:critgi}

In this section we will introduce a procedure that can be applied to compute the algebraic structure of the sandpile groups of the family of plane graphs that have a common weak dual graph in terms of the critical ideals of the common weak dual graph and the lengths of the cycles bounding the interior faces of a plane embedding.

The basic properties about critical ideals and determinantal ideals of graphs can be found in \cite{akm,corrval}, and in \cite{alflin} can be found other applications of the critical ideals not considered there.
Next we state few properties of the critical ideals. 
By convention $I_k(G)=\left< 1\right>$ if $k<1$, and $I_k(G)=\left< 0\right>$ if $k>n$.
An ideal is called {\it trivial} or {\it unit} if it is equal to $\lrangle{1}$.
The {\it algebraic co-rank} of $G$, denoted by $\gamma(G)$, is the number of critical ideals of $G$ equal to $\left<1\right>$.
It is known that if $i\leq j$, then $I_j(G) \subseteq I_i(G)$.
Furthermore, if $H$ is an induced subgraph of $G$, then $I_i(H)\subseteq I_i(G)$, from which follows that $\gamma(H)\leq \gamma(G)$.

The classic relation between critical ideals and the invariant factors of the sandpile groups of graphs are depicted by the following results.
First, we recall an alternative way to compute the invariant factors of integer matrices derived from the adjacency matrix.

\begin{lemma}\label{prop:evalmultiplevariables}\cite[Proposition 14]{akm}
Let $G$ be a graph with $n$ vertices and the indeterminates of $X=(x_1,\dots,x_n)$ are associated with the vertices of $G$. 
Let $M={a}I_n-A(G)$, where ${a}\in \mathbb{Z}^n$.
Then, the ideal in $\mathbb{Z}$ obtained from the evaluation of $I_k(G)$ at $X={a}$ is generated by $\Delta_k(M)$, that is, the $\gcd$ of the $k$-minors of the matrix $M$.
\end{lemma}

This result is convenient since the $k$-{\it th} invariant factor $d_k(M)$ of the SNF of $M$ is equal to $\frac{\Delta_k(M)}{\Delta_{k-1}(M)}$.
In particular, we can apply Lemma~\ref{prop:evalmultiplevariables} to the Laplacian matrix and reduced Laplacian matrix to give a method to compute the sandpile groups of some families of graphs.

\begin{proposition}\cite{corrval}
\label{cor:evaluations}
Let $G$ be a graph with vertex set $\{v_1,\dots,v_n\}$. 
Then,
\begin{enumerate}
    \item if $ \deg(G)=(\deg_G(v_1),\dots,\deg_G(v_n))$, then the $k$-{\it th} critical ideal of $G$ evaluated at $X=\deg(G)$ is generated by $\Delta_k(L(G))$, and $\gamma(G)\leq f_1(G)$,
    \item let $H$ be the graph constructed from $G$ by adding a new vertex $v_{n+1}$, and let ${ m}\in \mathbb{N}^n$, where ${ m}_i$ is the number of edges between $v_{n+1}$ and $v_i$, then the $k$-{\it th} critical ideal of $G$ evaluated at $X=\deg(G)+{ m}$ is generated by $\Delta_k(L_{n+1}(H))$, and $\gamma(G)\leq f_1(H)$.
\end{enumerate}
\end{proposition}
\begin{proof}
It follows from Lemma~\ref{prop:evalmultiplevariables}, note that in case (1) the evaluation of $A_X(G)$ at $X=\deg(G)$ equals $L(G)$. Moreover, note that $\Delta_j(L(G))=1$ for all $1\leq j\leq \gamma(G)$, therefore the first $\gamma(G)$ invariant factors are $1$. 
In case (2) the evaluation of $A_X(G)$ at $X=\deg(G)+m$ equals $L_{n+1}(H)$ and similarly to case (1) we have that $ f_1(H)\geq \gamma(G)$ 
\end{proof}

The next example will illustrate how the critical ideals can be used to compute the sandpile group of the family of graphs obtained from a graph $G$ by adding a new vertex $v$ with an arbitrary number of edges between $v$ and the vertices of $G$.

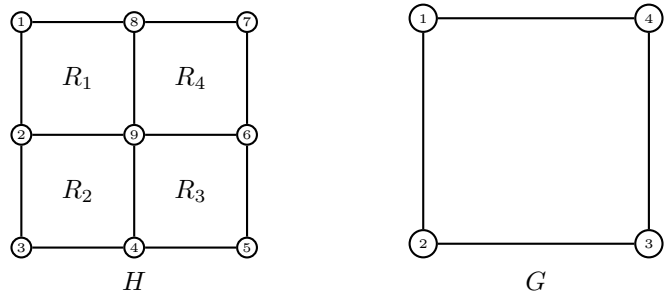
\begin{figure}[ht]
    \centering
    \begin{tabular}{c@{\extracolsep{2cm}}c}
    \begin{tikzpicture}[scale=1.5,thick]
		\tikzstyle{every node}=[minimum width=0pt, inner sep=1pt, circle]
			\draw (0,0) node[draw] (0) { \tiny 9};
			\draw (-1,1) node[draw] (1) { \tiny 1};
			\draw (-1,0) node[draw] (2) { \tiny 2};
			\draw (-1,-1) node[draw] (3) { \tiny 3};
			\draw (0,-1) node[draw] (4) { \tiny 4};
			\draw (1,-1) node[draw] (5) { \tiny 5};
			\draw (1,0) node[draw] (6) { \tiny 6};
			\draw (1,1) node[draw] (7) { \tiny 7};
			\draw (0,1) node[draw] (8) { \tiny 8};
			\draw (-0.5,0.5) node[] (9) { $R_1$};
			\draw (-0.5,-0.5) node[] (10) { $R_2$};
			\draw (0.5,-0.5) node[] (11) { $R_3$};
			\draw (0.5,0.5) node[] (12) { $R_4$};
			\draw  (0) edge (4);
			\draw  (0) edge (2);
			\draw  (1) edge (2);
			\draw  (2) edge (3);
			\draw  (3) edge (4);
			\draw  (4) edge (5);
			\draw  (5) edge (6);
			\draw  (6) edge (7);
			\draw  (6) edge (0);
			\draw  (7) edge (8);
			\draw  (8) edge (1);
			\draw  (8) edge (0);
		\end{tikzpicture}
    & 
    \begin{tikzpicture}[scale=3,thick]
		\tikzstyle{every node}=[minimum width=0pt, inner sep=2pt, circle]
			\draw (-0.5,0.5) node[draw] (9) { \tiny 1};
			\draw (-0.5,-0.5) node[draw] (10) { \tiny 2};
			\draw (0.5,-0.5) node[draw] (11) { \tiny 3};
			\draw (0.5,0.5) node[draw] (12) { \tiny 4};
			\draw  (9) edge (12);
			\draw  (10) edge (9);
			\draw  (11) edge (10);
			\draw  (12) edge (11);
		\end{tikzpicture}\\
		$H$ & $G$
    \end{tabular}
    \caption{A plane graph $H$ with 4 interior faces and its weak dual $G$.}
    \label{fig:exa}
\end{figure}

\begin{example}\label{exa:coneC8}
Let $H$ be the plane graph shown in Figure~\ref{fig:exa}.
Let $C_8$ be the cycle with 8 vertices obtained from $H$ by removing vertex $v_9$ and the edges incident to it.
The algebraic co-rank of $C_8$ is 6, and for the next critical ideal we will give their Gr\"obner bases since we need a simple basis that describe the ideal.
The Gr\"obner basis of the 7-th critical ideal of $C_8$ is generated by the following 3 polynomials:
\begin{eqnarray*}
p_1=x_1 + x_3 x_4 x_5 x_6 x_7 - x_3 x_4 x_5 - x_3 x_4 x_7 - x_3 x_6 x_7 + x_3 - x_5 x_6 x_7 + x_5 + x_7,\\
p_2=x_2 + x_4 x_5 x_6 x_7 x_8 - x_4 x_5 x_6 - x_4 x_5 x_8 - x_4 x_7 x_8 + x_4 - x_6 x_7 x_8 + x_6 + x_8,\\
p_3=x_3 x_4 x_5 x_6 x_7 x_8 - x_3 x_4 x_5 x_6 - x_3 x_4 x_5 x_8 - x_3 x_4 x_7 x_8 + \\
x_3 x_4 - x_3 x_6 x_7 x_8 + x_3 x_6 + x_3 x_8 - x_5 x_6 x_7 x_8 + x_5 x_6 + x_5 x_8 + x_7 x_8.
\end{eqnarray*}
The 8-th critical ideal of $C_8$ is generated by the determinant of $A_X(C_8)$:
\begin{eqnarray*}\tiny 
 x_1  x_2  x_3  x_4  x_5  x_6  x_7  x_8 -  x_1  x_2  x_3  x_4  x_5  x_6 -  x_1  x_2  x_3  x_4  x_5  x_8 -  x_1  x_2  x_3  x_4  x_7  x_8 \\
 +  x_1  x_2  x_3  x_4 -  x_1  x_2  x_3  x_6  x_7  x_8 +  x_1  x_2  x_3  x_6 +  x_1  x_2  x_3  x_8 -  x_1  x_2  x_5  x_6  x_7  x_8 \\
 +  x_1  x_2  x_5  x_6 +  x_1  x_2  x_5  x_8 +  x_1  x_2  x_7  x_8 -  x_1  x_2 -  x_1  x_4  x_5  x_6  x_7  x_8 +  x_1  x_4  x_5  x_6 \\
 +  x_1  x_4  x_5  x_8 +  x_1  x_4  x_7  x_8 -  x_1  x_4 +  x_1  x_6  x_7  x_8 -  x_1  x_6 -  x_1  x_8 -  x_2  x_3  x_4  x_5  x_6  x_7 \\
 +  x_2  x_3  x_4  x_5 +  x_2  x_3  x_4  x_7 +  x_2  x_3  x_6  x_7 -  x_2  x_3 +  x_2  x_5  x_6  x_7 -  x_2  x_5 -  x_2  x_7 \\
 -  x_3  x_4  x_5  x_6  x_7  x_8 +  x_3  x_4  x_5  x_6 +  x_3  x_4  x_5  x_8 +  x_3  x_4  x_7  x_8 -  x_3  x_4 +  x_3  x_6  x_7  x_8 -  x_3  x_6 \\
 -  x_3  x_8 +  x_4  x_5  x_6  x_7 -  x_4  x_5 -  x_4  x_7 +  x_5  x_6  x_7  x_8 -  x_5  x_6 -  x_5  x_8 -  x_6  x_7 -  x_7  x_8.
\end{eqnarray*}
In particular, by evaluating the polynomials $p_1,p_2,p_3$ and $\det(A_X(C_8))$ at $X=\deg(C_8)+(0,1,0,1,0,1,0,1)$, we obtain $\Delta_7(L_9(H))=\gcd(32,48,72)=8$, and $\Delta_8(L_9(H))=192$.
From which follows that the sandpile group $K(H)$ is isomorphic to $\mathbb{Z}_{8}\oplus\mathbb{Z}_{24}$.
\end{example}

The Gr\"obner basis for
the critical ideals of the complete graphs, the cycles and the paths were computed in \cite{corrval}.
In \cite{corrval1}, it was given a description of the generators of the $k$-{\it th}-critical ideal of any tree in terms of a set of special 2-matchings.
The generators of the critical ideals of other graph families have been computed in \cite{al,acv,yibo}.

A new relation is explored next based on the cycle-intersection matrix $C(H)$ of a plane graph $H$.

\begin{theorem}\label{thm:criticalideals_ciclematrix}
Let $G$ be a graph with vertex set $\{v_1,\dots,v_n\}$. 
If $G$ is the weak dual of the plane graph $H$ and ${ c}\in \mathbb{N}^n$ is such that ${ c}_i$ is the length of the cycle bounding the $i$-{\it th} finite face, then the ideal in $\mathbb{Z}$ obtained from the evaluation of $I_k(G)$ at  $X={ c}$ is generated by $\Delta_k(C(H))$.
And $f_1(C(H))\geq\gamma(G)$.
\end{theorem}
\begin{proof}
We have that $G=H_*$. Let us assume that $v_{n+1}\in H^*$ is the vertex that corresponds to the outer face of $H$. Then $C(H)$ is the reduced Laplacian matrix $L_{n+1}(H^*)$. Now, set $c=\deg(G)+{ m}$, where ${ m}_i$ is the number of edges between the vertex associated with the $i$-{\it th} interior face and the outer face. Thus the result follows by applying Proposition~\ref{cor:evaluations}.2.
\end{proof}

Let $G$ be a plane graph.
Therefore by Lemma~\ref{lemma:lemmacycleintersection} and Theorem~\ref{thm:criticalideals_ciclematrix}, 
the sandpile group of any plane graph $H$ having $G$ as weak dual can be obtained from the critical ideals of $G$ by evaluating the indeterminates $X=(x_1,\dots,x_n)$ at the lengths $c=({ c}_1,\dots,{ c}_n)$ of the cycles bounding the interior faces of $H$.
Also $\det(A_X(G))|_{X={ c}}=\tau(H)$.
Let us illustrate this with the following example.

\begin{example}
Let $G$ be the graph described in the right-hand side in Figure~\ref{fig:exa}. Then
\[
A_Y(G)=
\begin{bmatrix}
    y_1 & -1 & 0 & -1\\
-1 & y_2 & -1 & 0\\
 0 & -1 & y_3 & -1\\
-1 & 0 & -1 & y_4\\
\end{bmatrix}.
\]
Since there are 2-minors in $A_Y(G)$ equal to $\pm1$, then $\gamma(G)\geq 2$, the equality follows since the third critical ideal of $G$ is non-trivial.
The Gr\"obner basis of $I_3(G)$ is
\[ 
\lrangle{y_1 +  y_3,  y_2 +  y_4,  y_3y_4}
\]
Moreover, $I_4(G)=\lrangle{\det(A_Y(G))}=\lrangle{
y_1y_2y_3y_4 -  y_1y_2 -  y_1y_4 -  y_2y_3 -  y_3y_4}$.
Now, we will use these critical ideals to obtain the sandpile groups of any plane graph $H$ whose weak dual is isomorphic to $G$.
Thus, we only need to evaluate the indeterminates at the length of the cycles bounding the interior faces of $H$.
Note that the length of the interior faces of $H$ is at least 2 and at least one of the interior faces has length at least 3.
One of such cases is when all interior faces of $H$ have the same length, say $t$. 
Hence, for this case, $\Delta_3(C(H))=\gcd(2t,t^2)$ and $\Delta_4(C(H))=|t^4-4t^2|$.
It is not difficult to see that $\Delta_3(C(H))$ is equal to $t$ whenever $t$ is odd and it is equal to $2t$ whenever $t$ is even.
Therefore, if the interior faces of $H$ have the same length $t$, the sandpile group $K(H)$ of $H$ is isomorphic to $\mathbb{Z}_{\gcd(2t,t^2)}\oplus\mathbb{Z}_{\frac{|t^4-4t^2|}{\gcd(2t,t^2)}}$ and $\tau(H)=|t^4-4t^2|$.
Since $t\geq3$, then the sandpile group of $H$ is not cyclic.
\end{example}

\section{Sandpile groups of outerplanar graphs}\label{sec:outerplanar}

We call a graph {\it outerplanar} if it has a planar embedding with the outer face containing all the vertices.
An outerplanar graph equipped with such embedding is known as {\it outerplane graph}.

\begin{lemma}\cite{fgh}\label{lemma:weakdualouterplanar}
A graph $G$ is outerplanar if and only if it has a weak dual $G_*$ which is a forest.
\end{lemma}

One advantage of the outerplane graphs is that when the outerplanar has been embedded in the plane with all all the vertices lying on the outer face, then the weak dual is the union of the weak duals of the blocks of $G$. 

Next result implies that we should focus in computing sandpile groups of biconnected outerpanar graphs.

\begin{lemma}\cite{watkins}
Let $G$ be a graph with $b$ non-trivial blocks $B_1,\dots,B_b$.
Then $K(G)\cong K(B_1)\oplus\cdots\oplus K(B_b)$.
\end{lemma}

The following result is an specialization of Lemma~\ref{lemma:weakdualouterplanar}.

\begin{corollary}
A graph $G$ is biconnected outerplane if and only if its weak dual $G_*$ is a tree.
\end{corollary}

Now we will give a description of the generators of the critical ideals of any tree $T$, which were obtained in \cite{corrval1} in terms of the 2-matchings of the graph $T^l$, where $T^l$ is the graph obtained from $T$ by adding a loop at each vertex of $T$.

Recall that a \emph{$2$-matching} is a set of edges $\mathcal{M}\subseteq E(G)$ such that every vertex of $G$ is incident to at most two edges in $\mathcal{M}$ and note that a loop counts as two incidences for its respective vertex.
The set of 2-matchings of $T^l$ with $k$ edges is denoted by $2M(T^l,k)$.
Given a 2-matching $\mathcal{M}$ of $T^l$, the loops $\ell(M)$ of $\mathcal{M}$ is the edge set $\mathcal{M}\cap\{uu : u\in V(G)\}$. 
A 2-matching $\mathcal{M}$ of $T^l$ is {\it minimal} if there does not exist a 2-matching $\mathcal{M}'$ of $T^l$ such that $\ell(\mathcal{M}')\subsetneq \ell(\mathcal{M})$ and 
$|\mathcal{M}'|= \mathcal{M}$.
The set of minimal 2-matchings of $T^l$ will be denoted by $2M^*\left(T^l\right)$, and the set of minimal 2-matchings of $T^l$ with $k$ edges will be denoted by $2M^*_k\left(T^l\right)$.
Let $d_X(\mathcal{M})$ denote $\det(A_X(T)[V(\ell(\mathcal{M}))])$, that is, the determinant of the submatrix of $A_X(T)$ formed by selecting the columns and rows associated with the loops of $\mathcal{M}$.

\begin{lemma}\cite[Theorem 3.7]{corrval1}\label{thm:critidealstrees}
Let $T$ be a tree with $n$ vertices.
Then 
\[
I_k(T)=\lrangle{ \left\{ d_X(\mathcal{M}) : \mathcal{M}\in 2M^*_k\left(T^l\right)\right\} },
\]
for $k\in[n]$.
\end{lemma}

It follows directly from Theorem~\ref{thm:criticalideals_ciclematrix} and Lemma~\ref{thm:critidealstrees} that the sandpile groups of outerplanar graphs are determined in terms of the length of the cycles bounding the interior faces of their outerplane embeddings and the 2-matching of the weak dual with loops.

\begin{theorem}\label{main}
Let $G$ be a biconnected outerplane graph whose weak dual is the tree $T$ with $n$ vertices, and let $c=(c_1,\dots,c_n)$ be the vector of the lengths of the cycles bounding the finite faces $F_1,\dots,F_n$.
Let
\[
\Delta_k=\gcd\left( \left\{ d_X(\mathcal{M})|_{X=c} : \mathcal{M}\in 2M^*_k\left(T^l\right)\right\} \right),
\]
for $k\in[n]$.
Then $K(G)\cong\mathbb{Z}_{\Delta_1}\oplus\mathbb{Z}_\frac{\Delta_2}{\Delta_1}\oplus\cdots\mathbb{Z}_\frac{\Delta_n}{\Delta_{n-1}}$ and $\tau(G)=\Delta_n$.
\end{theorem}
Let us illustrate the utility of Theorem \ref{main} in the following example.
\begin{example}
\begin{figure}[ht]
    \centering
    \begin{tabular}{c@{\extracolsep{0.5cm}}c}
    \begin{tikzpicture}[scale=2,thick]
		\tikzstyle{every node}=[minimum width=0pt, inner sep=1pt, circle]
		    \draw (-2,1.2) node[draw] (0) { \tiny 1};
		    \draw (-1,0.75) node[draw] (1) { \tiny 2};
		    \draw (-0.3,1) node[draw] (2) { \tiny 3};
		    \draw (0.8,1.2) node[draw] (3) { \tiny 4};
		    \draw (1.4,0) node[draw] (4) { \tiny 5};
		    \draw (0.8,-1.2) node[draw] (5) { \tiny 6};
		    \draw (-0.3,-1.2) node[draw] (6) { \tiny 7};
		    \draw (-1,-0.75) node[draw] (7) { \tiny 8};
		    \draw (-2,-1.2) node[draw] (8) { \tiny 9};
		    \draw (-2.5,-0.25) node[draw] (9) { \tiny 10};
		    \draw (-2.5,0.25) node[draw] (10) { \tiny 11};
		    
		    \draw (-1.9,0.8) node[] (11) { $F_1$};
			\draw (-1.9,-0.8) node[] (12) { $F_2$};
			\draw (-1.5,0) node[] (13) { $F_3$};
			\draw (0,0) node[] (14) { $F_4$};
			\draw (0.7,0.8) node[] (15) { $F_5$};
			\draw (0.7,-0.8) node[] (16) { $F_6$};
			\draw  (1) edge (10);
			\draw  (1) edge (7);
			\draw  (7) edge (9);
			\draw  (2) edge (4);
			\draw  (4) edge (6);
			
			\draw  (0) -- (1) -- (2) -- (3) -- (4) -- (5);
			\draw  (5) -- (6) -- (7) -- (8) -- (9) -- (10) -- (0);
		\end{tikzpicture}
     &
    \begin{tikzpicture}[scale=2,thick]
		\tikzstyle{every node}=[minimum width=0pt, inner sep=2pt, circle]
			\draw (-0.5,1.2) node[draw] (17) { \tiny 1};
			\draw (0.5,1.2) node[draw] (18) { \tiny 2};
			\draw (0,0.4) node[draw] (19) { \tiny 3};
			\draw (0,-0.3) node[draw] (20) { \tiny 4};
			\draw (-0.5,-1.1) node[draw] (21) { \tiny 5};
			\draw (0.5,-1.1) node[draw] (22) { \tiny 6};
			\draw  (19) edge (17) edge (18);
			\draw  (20) edge (19) edge (21) edge (22);
		\end{tikzpicture} \\
		$G$ & $T$\\
    \end{tabular}
    \caption{An outerplane graph $G$ with 6 interior faces and its weak dual $T$.}
    \label{fig:exa2}
\end{figure}
Let $G$ be the outerplane graph in figure \ref{fig:exa2}, then $G_*=T$ where the vertex $i\in V(T)$ corresponds to the face $F_i$ of $G$ for each $1\leq i\leq 6$. We will use Theorem \ref{main} to compute the sandpile group of $K(G)$. We need to compute $2M^*_k(T^l)$ for $1\leq k\leq 6$. First, note that if $T^l$ has  minimal $2$-matching of size $k$ without loops, then $I_k(T)=\langle 1 \rangle$. It is easy to see that this is the case for $k\leq 4$ and then $\Delta_1 = \Delta_2 = \Delta_3 = \Delta_4=1$. On the other hand, for $k=5$,
\[2M^*_5(T^l)=\begin{Bmatrix}
\{(11),(22),(33),(45),(46) \},&
\{(13),(23),(44),(55),(66) \},\\
\{(11),(55),(23),(34),(46) \},&\{(11),(66),(23),(34),(45) \},\\
\{(22),(55),(13),(34),(46) \},&\{(22),(66),(13),(34),(45) \}
\end{Bmatrix}.  \]
Therefore, by Lemma \ref{thm:critidealstrees}, \[I_5(T^l)=\langle x_1x_2x_3-x_1-x_2,\ x_4x_5x_6-x_5-x_6,\ x_1x_5,\ x_1x_6,\ x_2x_5,\ x_2x_6\rangle .\] 
Moreover, the $6$-th critical ideal of $T$ is generated by $\det(A_X(T))$;
\begin{eqnarray*}\tiny 
x_1x_2x_3x_4x_6x_5- x_1 x_2 x_3 x_5- x_1 x_2 x_3 x_6- x_1 x_2 x_6 x_5 - x_1 x_4 x_6 x_5\\- x_2 x_4 x_6 x_5 +x_1 x_5 + x_2 x_5  + x_1 x_6 + x_2 x_6 \end{eqnarray*}

Now, since $c=(3,3,4,5,3,3)$ and by Theorem \ref{main}, $\Delta_5 = \gcd(30,39,9)=3$, $\Delta_6=1089$ and thus $K(G)=\mathbb{Z}_3\oplus \mathbb{Z}_{363}$.
Note that we can easily compute the sandpile group of any graph with $T$ as its weak dual, using the corresponding cycle-lengths. 
\begin{figure}[ht]
    \centering
    \begin{tabular}{c@{\extracolsep{0.2cm}}c}
    \begin{tikzpicture}[scale=2,thick]
		\tikzstyle{every node}=[minimum width=0pt, inner sep=1pt, circle]
		    \draw (-1.8,1.2) node[draw] (0) {\tiny 1};
		    \draw (-1,0.75) node[draw] (1) {\tiny 2};
		    \draw (-0.2,1.2) node[draw] (2) { \tiny 3};
		    \draw (0.2,0) node[draw] (3) { \tiny 4};
		    \draw (-0.2,-1.2) node[draw] (4) { \tiny 5};
		    \draw (-1,-0.75) node[draw] (6) { \tiny 6};
		    \draw (-1.8,-1.2) node[draw] (7) { \tiny 7};
		    \draw (-2.2,0) node[draw] (8) { \tiny 8};

		    \draw (-1.6,0.8) node[] (11) { $F_1$};
			\draw (-1.6,-0.8) node[] (12) { $F_2$};
			\draw (-1.5,0) node[] (13) { $F_3$};
			\draw (-0.5,0) node[] (14) { $F_4$};
			\draw (-0.4,0.8) node[] (15) { $F_5$};
			\draw (-0.4,-0.8) node[] (16) { $F_6$};
			\draw  (1) edge (8);
			\draw  (1) edge (6);
			\draw  (6) edge (8);
			\draw  (1) edge (3);
			\draw  (3) edge (6);
			
			\draw  (0) -- (1) -- (2) -- (3) -- (4);
			\draw  (4) -- (6) -- (7) -- (8) -- (0);
		\end{tikzpicture}
     &
    \begin{tikzpicture}[scale=2,thick]
		\tikzstyle{every node}=[minimum width=0pt, inner sep=1pt, circle]
		    \draw (-2,1.2) node[draw] (0) { \tiny 1};
		    \draw (-1.4,0.75) node[draw] (1) { \tiny 2};
		    \draw (-0.9,0.85) node[draw] (a) { \tiny 2};
		    \draw (-0.4,1) node[draw] (2) { \tiny 3};
		    \draw (0.2,1.2) node[draw] (3) { \tiny 4};
		    \draw (0.8,0.25) node[draw] (4) { \tiny 5};
		    \draw (0.8,-0.25) node[draw] (17) { \tiny 5};
		    \draw (0.2,-1.2) node[draw] (5) { \tiny 6};
		    \draw (-0.7,-1.2) node[draw] (6) { \tiny 7};
		    \draw (-1.4,-0.75) node[draw] (7) { \tiny 8};
		    \draw (-2,-1.2) node[draw] (8) { \tiny 9};
		    \draw (-2.3,-0.25) node[draw] (9) { \tiny 10};
		    \draw (-2.3,0.25) node[draw] (10) { \tiny 11};
		    
		    \draw (-1.9,0.8) node[] (11) { $F_1$};
			\draw (-1.9,-0.8) node[] (12) { $F_2$};
			\draw (-1.5,0) node[] (13) { $F_3$};
			\draw (0,0) node[] (14) { $F_4$};
			\draw (0.2,0.8) node[] (15) { $F_5$};
			\draw (0.2,-0.85) node[] (16) { $F_6$};
			\draw  (4) edge (17);
			\draw  (1) edge (10);
			\draw  (a) edge (7);
			\draw  (7) edge (9);
			\draw  (2) edge (4);
			\draw  (17) edge (6);
			
			\draw  (0) -- (1) -- (a) -- (2) -- (3) -- (4);
			\draw  (17) -- (5) -- (6) -- (7) -- (8) -- (9) -- (10) -- (0);
		\end{tikzpicture}\\
		$G_1$ & $G_2$\\
    \end{tabular}
    \caption{An outerplane graph $G$ with 6 interior faces and its weak dual $T$.}
    \label{fig:exa2.2}
\end{figure}
For instance, some allowed edge contractions or vertex splittings of $G$ as in Figure \ref{fig:exa2.2}. Let $c_1=(3,3,3,3,3,3)$ and $c_2=(3,3,5,6,3,3)$ be the vectors of lengths of the cycles bounding the interior faces of $G_1$ and $G_2$ respectively. Then \[K(G_1)=\mathbb{Z}_{\gcd(39,48,9)}\oplus\mathbb{Z}_{\frac{1791}{\gcd(39,48,9)}}=\mathbb{Z}_3\oplus\mathbb{Z}_{597} \text{ and}\] \[K(G_2)=\mathbb{Z}_{\gcd(21,9)}\oplus\mathbb{Z}_{\frac{360}{\gcd(21,9)}}=\mathbb{Z}_3\oplus\mathbb{Z}_{120}.\]
\end{example}

\begin{remark}
Note that if $G$ is a biconnected outerplane graph with weak dual $T$. Then any subdivision of the non-chordal edges of $G$ is an outerplane graph with the same weak dual. Therefore, by Theorem \ref{main}, the algebraic structure of the sandpile groups of any such graph in the homeomorphism class of $G$ is decoded in the combinatorial structure of $T$.
\end{remark}

Moreover, if $G$ is a biconnected outerplane graph whose weak dual is the tree $T$, then $f_1(C(G))\geq\gamma(T)$.
Let $\nu_2(G)$ denote the \emph{$2$-matching number} of $G$ that is defined as the maximum number of edges of a $2$-matching of $G$.
It was proven in \cite{corrval1} that for any tree $T$, the equality $\gamma(T)=\nu_2(T)$ holds.
Later, in \cite{alflin} it was proven that $\nu_2(T)=n-\delta(T)$ for any tree $T$ on $n$ vertices, where the parameter $\delta(T)$ is defined as the maximum of $p-q$ such that by deleting $q$ vertices from $T$ the remaining graph becomes $p$ paths.
Since it was found a linear-time algorithm for finding $\delta(T)$ \cite{JS02}, it was concluded in \cite{alflin} that there is a polynomial time algorithm to compute the algebraic co-rank for trees.
Also, in \cite{alflin}, it was proved that for any tree $T$, the algebraic co-rank $\gamma(T)$ coincides with the {\it minimum rank} $\mr(T)$ of $T$ and with $\mz(T):=|V(T)|-Z(T)$, where $Z(T)$ denote the {\it zero-forcing number} of $T$.

In the following the sandpile groups of some outerplanar graphs are further simplified.

\subsection{Outerplane graphs whose weak dual is a path}

Let us consider the outerplane graphs whose common weak dual is a path.
Let $(k_1,\dots,k_n)$ be a sequence of integers where each $k_i\geq 2$.
Let $PC_0$ denote the path with one edge. 
For each $1\leq i \leq n$, take the graph $PC_i$ from the graph $PC_{i-1}$ by adding a path with $k_{i}-1$ edges between any pair of adjacent vertices of the path added in the construction of $PC_{i-1}$.
Thus, the graph $PC_n$ consists of a stack of $n$ polygons with $k_1,\dots, k_n$ sides. The graph $PC_n$ is known as {\it polygon chain}.
Polygon chains are the outerplanar graphs having the path as a weak dual.

It is not difficult to see that $\gamma(G)=n-1$ if $G$ is a path with $n$ vertices.
The opposite is also true, see \cite[Corollary 3.9]{corrval1}.
From which follows that polygon chains have cyclic sandpile group.
The last critical ideal $I_n(P_n)$ of the path $P_n$ with $n$ vertices is generated by the determinant of $A_X(P_n)$.
Next relations follows directly from the determinant of $A_X(P_n)$. These were already noticed in \cite{BG,cm,Kerpkiy}.

\begin{lemma}\label{recursivelemma}
Let $P_n$ be the path with $n$ vertices and let $X=\{x_1,\dots,x_n\}$ a set of indeterminates associated with the vertices of $P_n$.
Then \[
\det(A_X(P_n))=x_n\det(A_X(P_{n-1}))-\det(A_X(P_{n-2}))\]
and $\tau(PC_n)=k_n\tau(PC_{n-1})-\tau(PC_{n-2})$.
\end{lemma}

In \cite{corrval}, an explicit computation of the determinant of $A_X(P_n)$ was obtained in terms of the matchings.

\begin{lemma}\cite[Corollary 4.5]{corrval}
Let $P_n$ be the path with $n$ vertices.
Then $\det(A_X(P_n))=\sum_{\mu\in M(P_n)}(-1)^{|\mu|}\prod_{v\notin V(\mu)}x_v$, where $M(P_n)$ is the set of matchings of $P_n$.
\end{lemma}

Next result follows directly from previous lemma and Theorem~\ref{thm:criticalideals_ciclematrix}.

\begin{theorem}\label{coro:PC_n}
Let $PC_n$ be a polygon chain whose stack of polygons have $k_1,\dots,k_n$ sides.
Then the sandpile group $K(PC_n)$ of $PC_n$ is cyclic of order $$\tau(PC_n)=\sum_{\mu\in M(P_n)}(-1)^{|\mu|}\prod_{v\notin V(\mu)}k_v,$$ where $M(P_n)$ is the set of matchings of $P_n$.
\end{theorem}
Now we proceed to analyze an special family of polygon chains. A polygon chain is called a \textit{polygon ladder} 
if each of its polygons has the same number of sides.

\begin{example}\label{exa:3.1}
Let $PL_n^k$ be the polygon ladder consisting of $n$ $k$-polygons with $k\geq 3$. By Theorem \ref{coro:PC_n} its sandpile group is cyclic of order
$$\tau(PL_n^k)=\sum_{\mu\in M(P_n)}(-1)^{|\mu|}\prod_{v\notin V(\mu)}k.$$
Let $\nu (G)$ be the matching number of $G$. It is easy to check that the number of matchings of $P_n$ of size $i$ is $\binom{n-i}{i}$ for $i=1,\ldots,\nu (P_n)$. 
If $n$ is even, say $n=2m$ for some positive integer, then $\nu (P_n) =m$. 
Similarly, when $n$ is odd. 
Assume $n=2m+1$ for some positive integer $m$, then $\nu (P_n) =m$. 
In both cases the matching number of $P_n$ is $\lfloor n/2 \rfloor$. 
Therefore,
$$\tau(PL_n^k)=\sum_{i=0}^{\lfloor n/2 \rfloor} (-1)^{i} \binom{n-i}{i}  k^{n-2i}, \text{ for } n\geq 1.$$
Since $0<\frac{4}{k^2}<1$, we have that\[\tau(PL_n^k) = k^n\, _2 F_1\left(\frac{1}{2}-\frac{n}{2}, -\frac{n}{2};-n;\frac{4}{k^2}\right) \text{ for }n\geq 1, \]
where $_2 F_1(a,b;c;x)$ is the Gauss's hypergeometric function.
Next we present three more specific instances. 
First, let us address the case of $PL_n^4=P_2\Box P_n$ (also known as the \textit{ladder graph} or the $2\times n$ grid). 
We have that $K(PL_n^4)$ is a cyclic group of order
$$\tau(PL_n^4) = \frac{1}{2\sqrt{3}} \big( (2+\sqrt{3})^{n+1} - (2-\sqrt{3})^{n+1} \big), \text{ for } n\geq 1 .$$ 
On the other hand, consider $PL_n^6$ (also called as an \textit{hexagonal chain}). Hence $K(PL_n^6)$ is a cyclic group of order
$$\tau(PL_n^6) = \frac{1}{4\sqrt{2}} \big( (3+2\sqrt{2})^{n+1} - (3-2\sqrt{2})^{n+1} \big), \text{ for } n\geq 1.$$
Lastly, consider the polygonal ladder with $n$ octagons $PL_n^8$. 
In this case we have that
$$\tau(PL_n^8) = \frac{1}{2\sqrt{15}} \big( (4+\sqrt{15})^{n+1} - (4-\sqrt{15})^{n+1} \big), \text{ for } n\geq 1.$$
In Table \ref{tab:polygonladdertable} we list the value of $\big|K(PL_n^k)\big|$ for  $k=4,6,8$ and $1\leq n\leq 11$.
\begin{table}[h]
    \centering
    \begin{tabular}{|c|c|c|c|}
        \hline
        $n$ & $\tau(PL_n^4)$ & $\tau(PL_n^6)$ & $\tau(PL_n^8)$\\
        \hline
        $1$ & $4$ & $6$ & $8$\\
        \hline
        $2$&$15$&$35$&$63$\\
        \hline
        $3$&$56$&$204$&$496$\\
        \hline
        $4$&$209$&$1189$&$3905$\\
        \hline
        $5$&$780$&$6930$&$30744$\\
        \hline
        $6$&$2911$&$40391$&$242047$\\
        \hline
        $7$&$10864$&$235416$&$1905632$\\
        \hline
        $8$&$40545$&$1372105$&$15003009$\\
        \hline
        $9$&$151316$&$7997214$&$118118440$\\
        \hline
        $10$&$564719$&$46611179$&$929944511$\\
        \hline
        $11$&$2107560$&$271669860$ &$7321437648$\\
        \hline
    \end{tabular}
    \caption{$\tau(PL_n^4)$, $\tau(PL_n^6)$ and $\tau(PL_n^8)$ for $1\leq n \leq 11$}
    \label{tab:polygonladdertable}
\end{table}
\end{example}

\subsection{Outerplane graphs whose weak dual is a starlike tree}

We denote by $S(n_1,\dots,n_l)
$ a {\it starlike
tree} in which removing the central vertex leaves disjoint paths $P_{n_1}$,\dots,$P_{n_l}$ in which exactly one endpoint of each path is a leaf on $S(n_1,\dots,n_l)
$.

Let $C_l= v_1 e_1 v_2 e_2 \cdots v_l e_l v_1$ be a cycle of length $l$, and $PC_{n_1},\dots,PC_{n_l}$ be $l$ polygon chains. 
A {\it polygon flower} $F = F (C_l ; PC_{n_1},\dots, PC_{n_l} )$ is constructed by identifying, for $i\in[l]$, the edges $e_i\in C_l$ and $e'_i\in PC_{n_i}$ such that $e'_i$ is in the first or the last polygon of $PC_{n_i}$ and is not contained in another polygon of this polygon chain.
The weak dual of an outerplane embedding of polygon flowers are starlike trees. The number of spanning trees of $F$ is closely related to the number of spanning trees of its polygon chains

\begin{theorem}\cite[Corollary 4.2]{cm}\label{cor:bojan}\label{tauF}
Let $F = F(C_l; PC_{n_1}, . . . , PC_{n_l})$ be a polygon flower. Then
\[ \tau (F)=\left( \prod_{j=1}^l \tau (PC_{n_j})\right) \sum_{i=1}^l \frac{\tau (PC_{n_i}/ e_i) }{\tau (PC_{n_i}) } \]
where $PC_{n_i}/ e_i$ denotes the graph obtained from $PC_{n_i}$ by contracting the edge $e_i$.
\end{theorem}

Moreover, in \cite{cm} the sandpile group of the polygon flowers were obtained in terms of the spanning tree numbers of the polygon chains.

\begin{lemma}\cite[Theorem 4.3]{cm}\label{teo:bojan}
Let $F=F(C_l;PC_{n_1},\dots, PC_{n_l})$ be a polygon flower. 
For $j\in[l-2]$, $\Delta_j=\gcd( \tau (PC_{n_{i_1}} ) \cdots \tau(PC_{n_{i_j}} ) : 1 \leq i_1 < \cdots < i_j \leq l)$. Then
\[
K(F) = \mathbb{Z}_{\Delta_1} \oplus \mathbb{Z}_\frac{\Delta_2}{\Delta_1} \oplus\cdots\oplus\mathbb{ Z}_\frac{\Delta_{t-2}}{\Delta_{t-3}}\oplus\mathbb{Z}_\frac{\tau(F)}{\Delta_{t-2}}.
\]
\end{lemma}

By using Lemma~\ref{teo:bojan} and Theorem~\ref{coro:PC_n}, we can obtain an equivalent result stated in terms of matchings of the path and the length of the polygons.

\begin{theorem}
Let $F=F(C_l;PC_{n_1},\dots, PC_{n_l})$ be a polygon flower, where  $k^i_1,\dots,k^i_{n_i}$ are the sizes of the polygons of $PC_{n_i}$. 
Let 
\[
\omega(n_i,k^i_1,\dots,k^i_{n_i})=\sum_{\mu\in M(P_{n_i})}(-1)^{|\mu|}\prod_{v\notin V(\mu)}k^i_v.
\]
For $j\in[l-2]$, $\Delta_j=\gcd( \omega(n_{i_1},k^{i_1}_1,\dots,k^{i_1}_{n_{i_1}}) \cdots \omega(n_{i_j},k^{i_j}_1,\dots,k^{i_j}_{n_{i_j}}) : 1 \leq i_1 < \cdots < i_j \leq l)$. Then
\[
K(F) = \mathbb{Z}_{\Delta_1} \oplus \mathbb{Z}_\frac{\Delta_2}{\Delta_1} \oplus\cdots\oplus\mathbb{ Z}_\frac{\Delta_{t-2}}{\Delta_{t-3}}\oplus\mathbb{Z}_\frac{\tau(F)}{\Delta_{t-2}}.
\]
\end{theorem}
Finally, we complement Example \ref{exa:3.1} analyzing a certain polygon flower constructed with polygon ladders.
\begin{example}
Let $F=F(C_5;PC_{n_1},PC_{n_2},PC_{n_3},PC_{n_4},PC_{n_5})$ and set the polygon chains of $F$ as $PC_{n_1}=PL_5^4,\ PC_{n_2}=PL_8^4,\ PC_{n_3}=PL_2^6,\ PC_{n_4}=PL_5^6\text{ and }PC_{n_5}=PL_5^8$. 
Moreover, if $1\leq j\leq 5$ and $\tau (PC_{n_j}) = \tau (PL_n^k)$ with $n\geq 2$ and $k\geq 3$, by Lemma \ref{recursivelemma} we have that \[\tau (PC_{n_j}/e_j) = (k-1)\tau (PL_{n-1}^k)-\tau (PL_{n-2}^k)=\tau (PL_n^k)-\tau (PL_{n-1}^k)\] Therefore, by Theorem \ref{tauF} and using Table \ref{tab:polygonladdertable} $$ \tau (F)=\left( \prod_{j=1}^5 \tau (PC_{n_j})\right) \sum_{i=1}^5 \frac{\tau (PC_{n_i}/ e_i) }{\tau (PC_{n_i}) }$$
\[=(235827017145720000)\left( \frac{571}{780} + \frac{29681}{40545} + \frac{29}{35} + \frac{5741}{6930} + \frac{26839}{30744}\right).\]

Hence, $\Delta_1=1,\ \Delta_2=15,\, \Delta_3=9450$ and $\tau(F)=941912914331277000$. 
Thus the sandpile group of $F$ is $\mathbb{Z}_{15}\oplus \mathbb{Z}_{630}\oplus \mathbb{Z}_{99673324267860}$.
\end{example}

\section{Identity element of the sandpile group of outerplanar graphs}\label{sec:identity}
Throughout this section we will consider outerplane graphs to be biconnected unless otherwise stated.
Determining the combinatorial structure of the recurrent configurations for outerplanar graph seems to be a more challenging problem since it depends on the sink vertex, and the sandpile groups are not always cyclic.
However, we will consider the dual of an outerplane graph since the vertex associated with the outer face is a natural sink vertex, the weak dual is a tree and from a recurrent configuration of this dual graph we can recover the associated recurrent configurations of the outerplane graph with different sink vertices.

\begin{figure}[ht!]
    \centering
    \begin{tabular}{cccccccccc}
    
    \begin{tikzpicture}[scale=0.5,thick]
    \tikzstyle{every node}=[minimum width=0pt, inner sep=1pt, circle]
        \draw (0:1) node[draw] (0) { \tiny 0};
        \draw (180:1) node[draw] (1) { \tiny 1};
        \draw  (0) edge (1);
    \draw (0,-2) node () {$2_{0}$};
    \end{tikzpicture}
&
    \begin{tikzpicture}[scale=0.5,thick]
    \tikzstyle{every node}=[minimum width=0pt, inner sep=1pt, circle]
        \draw (0:1) node[draw] (0) { \tiny 0};
        \draw (120:1) node[draw] (1) { \tiny 1};
        \draw (240:1) node[draw] (2) { \tiny 2};
        \draw  (0) edge (1);
        \draw  (0) edge (2);
    \draw (0,-2) node () {$3_{0}$};
    \end{tikzpicture}
&
    \begin{tikzpicture}[scale=0.5,thick]
    \tikzstyle{every node}=[minimum width=0pt, inner sep=1pt, circle]
        \draw (0:1) node[draw] (0) { \tiny 0};
        \draw (90:1) node[draw] (1) { \tiny 1};
        \draw (180:1) node[draw] (2) { \tiny 2};
        \draw (270:1) node[draw] (3) { \tiny 3};
        \draw  (0) edge (1);
        \draw  (0) edge (3);
        \draw  (1) edge (2);
    \draw (0,-2) node () {$4_{0}$};
    \end{tikzpicture}
&
    \begin{tikzpicture}[scale=0.5,thick]
    \tikzstyle{every node}=[minimum width=0pt, inner sep=1pt, circle]
        \draw (0:1) node[draw] (0) { \tiny 0};
        \draw (90:1) node[draw] (1) { \tiny 1};
        \draw (180:1) node[draw] (2) { \tiny 2};
        \draw (270:1) node[draw] (3) { \tiny 3};
        \draw  (0) edge (1);
        \draw  (0) edge (2);
        \draw  (0) edge (3);
    \draw (0,-2) node () {$4_{1}$};
    \end{tikzpicture}
&
    \begin{tikzpicture}[scale=0.5,thick]
    \tikzstyle{every node}=[minimum width=0pt, inner sep=1pt, circle]
        \draw (0:1) node[draw] (0) { \tiny 0};
        \draw (72:1) node[draw] (1) { \tiny 1};
        \draw (144:1) node[draw] (2) { \tiny 2};
        \draw (216:1) node[draw] (3) { \tiny 3};
        \draw (288:1) node[draw] (4) { \tiny 4};
        \draw  (0) edge (1);
        \draw  (0) edge (3);
        \draw  (1) edge (2);
        \draw  (3) edge (4);
    \draw (0,-2) node () {$5_{0}$};
    \end{tikzpicture}
&
    \begin{tikzpicture}[scale=0.5,thick]
    \tikzstyle{every node}=[minimum width=0pt, inner sep=1pt, circle]
        \draw (0:1) node[draw] (0) { \tiny 0};
        \draw (72:1) node[draw] (1) { \tiny 1};
        \draw (144:1) node[draw] (2) { \tiny 2};
        \draw (216:1) node[draw] (3) { \tiny 3};
        \draw (288:1) node[draw] (4) { \tiny 4};
        \draw  (0) edge (1);
        \draw  (0) edge (3);
        \draw  (0) edge (4);
        \draw  (1) edge (2);
    \draw (0,-2) node () {$5_{1}$};
    \end{tikzpicture}
&
    \begin{tikzpicture}[scale=0.5,thick]
    \tikzstyle{every node}=[minimum width=0pt, inner sep=1pt, circle]
        \draw (0:1) node[draw] (0) { \tiny 0};
        \draw (72:1) node[draw] (1) { \tiny 1};
        \draw (144:1) node[draw] (2) { \tiny 2};
        \draw (216:1) node[draw] (3) { \tiny 3};
        \draw (288:1) node[draw] (4) { \tiny 4};
        \draw  (0) edge (1);
        \draw  (0) edge (2);
        \draw  (0) edge (3);
        \draw  (0) edge (4);
    \draw (0,-2) node () {$5_{2}$};
    \end{tikzpicture}
\\
    \begin{tikzpicture}[scale=0.5,thick]
    \tikzstyle{every node}=[minimum width=0pt, inner sep=1pt, circle]
        \draw (0:1) node[draw] (0) { \tiny 0};
        \draw (60:1) node[draw] (1) { \tiny 1};
        \draw (120:1) node[draw] (2) { \tiny 2};
        \draw (180:1) node[draw] (3) { \tiny 3};
        \draw (240:1) node[draw] (4) { \tiny 4};
        \draw (300:1) node[draw] (5) { \tiny 5};
        \draw  (0) edge (1);
        \draw  (0) edge (4);
        \draw  (1) edge (2);
        \draw  (2) edge (3);
        \draw  (4) edge (5);
    \draw (0,-2) node () {$6_{0}$};
    \end{tikzpicture}
&
    \begin{tikzpicture}[scale=0.5,thick]
    \tikzstyle{every node}=[minimum width=0pt, inner sep=1pt, circle]
        \draw (0:1) node[draw] (0) { \tiny 0};
        \draw (60:1) node[draw] (1) { \tiny 1};
        \draw (120:1) node[draw] (2) { \tiny 2};
        \draw (180:1) node[draw] (3) { \tiny 3};
        \draw (240:1) node[draw] (4) { \tiny 4};
        \draw (300:1) node[draw] (5) { \tiny 5};
        \draw  (0) edge (1);
        \draw  (0) edge (4);
        \draw  (1) edge (2);
        \draw  (1) edge (3);
        \draw  (4) edge (5);
    \draw (0,-2) node () {$6_{1}$};
    \end{tikzpicture}
&
    \begin{tikzpicture}[scale=0.5,thick]
    \tikzstyle{every node}=[minimum width=0pt, inner sep=1pt, circle]
        \draw (0:1) node[draw] (0) { \tiny 0};
        \draw (60:1) node[draw] (1) { \tiny 1};
        \draw (120:1) node[draw] (2) { \tiny 2};
        \draw (180:1) node[draw] (3) { \tiny 3};
        \draw (240:1) node[draw] (4) { \tiny 4};
        \draw (300:1) node[draw] (5) { \tiny 5};
        \draw  (0) edge (1);
        \draw  (0) edge (4);
        \draw  (0) edge (5);
        \draw  (1) edge (2);
        \draw  (1) edge (3);
    \draw (0,-2) node () {$6_{2}$};
    \end{tikzpicture}
&
    \begin{tikzpicture}[scale=0.5,thick]
    \tikzstyle{every node}=[minimum width=0pt, inner sep=1pt, circle]
        \draw (0:1) node[draw] (0) { \tiny 0};
        \draw (60:1) node[draw] (1) { \tiny 1};
        \draw (120:1) node[draw] (2) { \tiny 2};
        \draw (180:1) node[draw] (3) { \tiny 3};
        \draw (240:1) node[draw] (4) { \tiny 4};
        \draw (300:1) node[draw] (5) { \tiny 5};
        \draw  (0) edge (1);
        \draw  (0) edge (3);
        \draw  (0) edge (5);
        \draw  (1) edge (2);
        \draw  (3) edge (4);
    \draw (0,-2) node () {$6_{3}$};
    \end{tikzpicture}
&
    \begin{tikzpicture}[scale=0.5,thick]
    \tikzstyle{every node}=[minimum width=0pt, inner sep=1pt, circle]
        \draw (0:1) node[draw] (0) { \tiny 0};
        \draw (60:1) node[draw] (1) { \tiny 1};
        \draw (120:1) node[draw] (2) { \tiny 2};
        \draw (180:1) node[draw] (3) { \tiny 3};
        \draw (240:1) node[draw] (4) { \tiny 4};
        \draw (300:1) node[draw] (5) { \tiny 5};
        \draw  (0) edge (1);
        \draw  (0) edge (3);
        \draw  (0) edge (4);
        \draw  (0) edge (5);
        \draw  (1) edge (2);
    \draw (0,-2) node () {$6_{4}$};
    \end{tikzpicture}
&
    \begin{tikzpicture}[scale=0.5,thick]
    \tikzstyle{every node}=[minimum width=0pt, inner sep=1pt, circle]
        \draw (0:1) node[draw] (0) { \tiny 0};
        \draw (60:1) node[draw] (1) { \tiny 1};
        \draw (120:1) node[draw] (2) { \tiny 2};
        \draw (180:1) node[draw] (3) { \tiny 3};
        \draw (240:1) node[draw] (4) { \tiny 4};
        \draw (300:1) node[draw] (5) { \tiny 5};
        \draw  (0) edge (1);
        \draw  (0) edge (2);
        \draw  (0) edge (3);
        \draw  (0) edge (4);
        \draw  (0) edge (5);
    \draw (0,-2) node () {$6_{5}$};
    \end{tikzpicture}
&
    \begin{tikzpicture}[scale=0.5,thick]
    \tikzstyle{every node}=[minimum width=0pt, inner sep=1pt, circle]
        \draw (0:1) node[draw] (0) { \tiny 0};
        \draw (360/7:1) node[draw] (1) { \tiny 1};
        \draw (720/7:1) node[draw] (2) { \tiny 2};
        \draw (1080/7:1) node[draw] (3) { \tiny 3};
        \draw (1440/7:1) node[draw] (4) { \tiny 4};
        \draw (1800/7:1) node[draw] (5) { \tiny 5};
        \draw (2160/7:1) node[draw] (6) { \tiny 6};
        \draw  (0) edge (1);
        \draw  (0) edge (4);
        \draw  (1) edge (2);
        \draw  (2) edge (3);
        \draw  (4) edge (5);
        \draw  (5) edge (6);
    \draw (0,-2) node () {$7_{0}$};
    \end{tikzpicture}
\\
    \begin{tikzpicture}[scale=0.5,thick]
    \tikzstyle{every node}=[minimum width=0pt, inner sep=1pt, circle]
        \draw (0:1) node[draw] (0) { \tiny 0};
        \draw (360/7:1) node[draw] (1) { \tiny 1};
        \draw (720/7:1) node[draw] (2) { \tiny 2};
        \draw (1080/7:1) node[draw] (3) { \tiny 3};
        \draw (1440/7:1) node[draw] (4) { \tiny 4};
        \draw (1800/7:1) node[draw] (5) { \tiny 5};
        \draw (2160/7:1) node[draw] (6) { \tiny 6};
        \draw  (0) edge (1);
        \draw  (0) edge (4);
        \draw  (1) edge (2);
        \draw  (2) edge (3);
        \draw  (4) edge (5);
        \draw  (4) edge (6);
    \draw (0,-2) node () {$7_{1}$};
    \end{tikzpicture}
&
    \begin{tikzpicture}[scale=0.5,thick]
    \tikzstyle{every node}=[minimum width=0pt, inner sep=1pt, circle]
        \draw (0:1) node[draw] (0) { \tiny 0};
        \draw (360/7:1) node[draw] (1) { \tiny 1};
        \draw (720/7:1) node[draw] (2) { \tiny 2};
        \draw (1080/7:1) node[draw] (3) { \tiny 3};
        \draw (1440/7:1) node[draw] (4) { \tiny 4};
        \draw (1800/7:1) node[draw] (5) { \tiny 5};
        \draw (2160/7:1) node[draw] (6) { \tiny 6};
        \draw  (0) edge (1);
        \draw  (0) edge (4);
        \draw  (0) edge (6);
        \draw  (1) edge (2);
        \draw  (2) edge (3);
        \draw  (4) edge (5);
    \draw (0,-2) node () {$7_{2}$};
    \end{tikzpicture}
&
    \begin{tikzpicture}[scale=0.5,thick]
    \tikzstyle{every node}=[minimum width=0pt, inner sep=1pt, circle]
        \draw (0:1) node[draw] (0) { \tiny 0};
        \draw (360/7:1) node[draw] (1) { \tiny 1};
        \draw (720/7:1) node[draw] (2) { \tiny 2};
        \draw (1080/7:1) node[draw] (3) { \tiny 3};
        \draw (1440/7:1) node[draw] (4) { \tiny 4};
        \draw (1800/7:1) node[draw] (5) { \tiny 5};
        \draw (2160/7:1) node[draw] (6) { \tiny 6};
        \draw  (0) edge (1);
        \draw  (0) edge (5);
        \draw  (1) edge (2);
        \draw  (1) edge (3);
        \draw  (1) edge (4);
        \draw  (5) edge (6);
    \draw (0,-2) node () {$7_{3}$};
    \end{tikzpicture}
&
    \begin{tikzpicture}[scale=0.5,thick]
    \tikzstyle{every node}=[minimum width=0pt, inner sep=1pt, circle]
        \draw (0:1) node[draw] (0) { \tiny 0};
        \draw (360/7:1) node[draw] (1) { \tiny 1};
        \draw (720/7:1) node[draw] (2) { \tiny 2};
        \draw (1080/7:1) node[draw] (3) { \tiny 3};
        \draw (1440/7:1) node[draw] (4) { \tiny 4};
        \draw (1800/7:1) node[draw] (5) { \tiny 5};
        \draw (2160/7:1) node[draw] (6) { \tiny 6};
        \draw  (0) edge (1);
        \draw  (0) edge (4);
        \draw  (1) edge (2);
        \draw  (1) edge (3);
        \draw  (4) edge (5);
        \draw  (4) edge (6);
    \draw (0,-2) node () {$7_{4}$};
    \end{tikzpicture}
&
    \begin{tikzpicture}[scale=0.5,thick]
    \tikzstyle{every node}=[minimum width=0pt, inner sep=1pt, circle]
        \draw (0:1) node[draw] (0) { \tiny 0};
        \draw (360/7:1) node[draw] (1) { \tiny 1};
        \draw (720/7:1) node[draw] (2) { \tiny 2};
        \draw (1080/7:1) node[draw] (3) { \tiny 3};
        \draw (1440/7:1) node[draw] (4) { \tiny 4};
        \draw (1800/7:1) node[draw] (5) { \tiny 5};
        \draw (2160/7:1) node[draw] (6) { \tiny 6};
        \draw  (0) edge (1);
        \draw  (0) edge (4);
        \draw  (0) edge (6);
        \draw  (1) edge (2);
        \draw  (1) edge (3);
        \draw  (4) edge (5);
    \draw (0,-2) node () {$7_{5}$};
    \end{tikzpicture}
&
    \begin{tikzpicture}[scale=0.5,thick]
    \tikzstyle{every node}=[minimum width=0pt, inner sep=1pt, circle]
        \draw (0:1) node[draw] (0) { \tiny 0};
        \draw (360/7:1) node[draw] (1) { \tiny 1};
        \draw (720/7:1) node[draw] (2) { \tiny 2};
        \draw (1080/7:1) node[draw] (3) { \tiny 3};
        \draw (1440/7:1) node[draw] (4) { \tiny 4};
        \draw (1800/7:1) node[draw] (5) { \tiny 5};
        \draw (2160/7:1) node[draw] (6) { \tiny 6};
        \draw  (0) edge (1);
        \draw  (0) edge (4);
        \draw  (0) edge (5);
        \draw  (0) edge (6);
        \draw  (1) edge (2);
        \draw  (1) edge (3);
    \draw (0,-2) node () {$7_{6}$};
    \end{tikzpicture}
&
    \begin{tikzpicture}[scale=0.5,thick]
    \tikzstyle{every node}=[minimum width=0pt, inner sep=1pt, circle]
        \draw (0:1) node[draw] (0) { \tiny 0};
        \draw (360/7:1) node[draw] (1) { \tiny 1};
        \draw (720/7:1) node[draw] (2) { \tiny 2};
        \draw (1080/7:1) node[draw] (3) { \tiny 3};
        \draw (1440/7:1) node[draw] (4) { \tiny 4};
        \draw (1800/7:1) node[draw] (5) { \tiny 5};
        \draw (2160/7:1) node[draw] (6) { \tiny 6};
        \draw  (0) edge (1);
        \draw  (0) edge (3);
        \draw  (0) edge (5);
        \draw  (1) edge (2);
        \draw  (3) edge (4);
        \draw  (5) edge (6);
    \draw (0,-2) node () {$7_{7}$};
    \end{tikzpicture}
\\
    \begin{tikzpicture}[scale=0.5,thick]
    \tikzstyle{every node}=[minimum width=0pt, inner sep=1pt, circle]
        \draw (0:1) node[draw] (0) { \tiny 0};
        \draw (360/7:1) node[draw] (1) { \tiny 1};
        \draw (720/7:1) node[draw] (2) { \tiny 2};
        \draw (1080/7:1) node[draw] (3) { \tiny 3};
        \draw (1440/7:1) node[draw] (4) { \tiny 4};
        \draw (1800/7:1) node[draw] (5) { \tiny 5};
        \draw (2160/7:1) node[draw] (6) { \tiny 6};
        \draw  (0) edge (1);
        \draw  (0) edge (3);
        \draw  (0) edge (5);
        \draw  (0) edge (6);
        \draw  (1) edge (2);
        \draw  (3) edge (4);
    \draw (0,-2) node () {$7_{8}$};
    \end{tikzpicture}
&
    \begin{tikzpicture}[scale=0.5,thick]
    \tikzstyle{every node}=[minimum width=0pt, inner sep=1pt, circle]
        \draw (0:1) node[draw] (0) { \tiny 0};
        \draw (360/7:1) node[draw] (1) { \tiny 1};
        \draw (720/7:1) node[draw] (2) { \tiny 2};
        \draw (1080/7:1) node[draw] (3) { \tiny 3};
        \draw (1440/7:1) node[draw] (4) { \tiny 4};
        \draw (1800/7:1) node[draw] (5) { \tiny 5};
        \draw (2160/7:1) node[draw] (6) { \tiny 6};
        \draw  (0) edge (1);
        \draw  (0) edge (3);
        \draw  (0) edge (4);
        \draw  (0) edge (5);
        \draw  (0) edge (6);
        \draw  (1) edge (2);
    \draw (0,-2) node () {$7_{9}$};
    \end{tikzpicture}
&
    \begin{tikzpicture}[scale=0.5,thick]
    \tikzstyle{every node}=[minimum width=0pt, inner sep=1pt, circle]
        \draw (0:1) node[draw] (0) { \tiny 0};
        \draw (360/7:1) node[draw] (1) { \tiny 1};
        \draw (720/7:1) node[draw] (2) { \tiny 2};
        \draw (1080/7:1) node[draw] (3) { \tiny 3};
        \draw (1440/7:1) node[draw] (4) { \tiny 4};
        \draw (1800/7:1) node[draw] (5) { \tiny 5};
        \draw (2160/7:1) node[draw] (6) { \tiny 6};
        \draw  (0) edge (1);
        \draw  (0) edge (2);
        \draw  (0) edge (3);
        \draw  (0) edge (4);
        \draw  (0) edge (5);
        \draw  (0) edge (6);
    \draw (0,-2) node () {$7_{10}$};
    \end{tikzpicture}
&
    \begin{tikzpicture}[scale=0.5,thick]
    \tikzstyle{every node}=[minimum width=0pt, inner sep=1pt, circle]
        \draw (0:1) node[draw] (0) { \tiny 0};
        \draw (45:1) node[draw] (1) { \tiny 1};
        \draw (90:1) node[draw] (2) { \tiny 2};
        \draw (135:1) node[draw] (3) { \tiny 3};
        \draw (180:1) node[draw] (4) { \tiny 4};
        \draw (225:1) node[draw] (5) { \tiny 5};
        \draw (270:1) node[draw] (6) { \tiny 6};
        \draw (315:1) node[draw] (7) { \tiny 7};
        \draw  (0) edge (1);
        \draw  (0) edge (5);
        \draw  (1) edge (2);
        \draw  (2) edge (3);
        \draw  (3) edge (4);
        \draw  (5) edge (6);
        \draw  (6) edge (7);
    \draw (0,-2) node () {$8_{0}$};
    \end{tikzpicture}
&
    \begin{tikzpicture}[scale=0.5,thick]
    \tikzstyle{every node}=[minimum width=0pt, inner sep=1pt, circle]
        \draw (0:1) node[draw] (0) { \tiny 0};
        \draw (45:1) node[draw] (1) { \tiny 1};
        \draw (90:1) node[draw] (2) { \tiny 2};
        \draw (135:1) node[draw] (3) { \tiny 3};
        \draw (180:1) node[draw] (4) { \tiny 4};
        \draw (225:1) node[draw] (5) { \tiny 5};
        \draw (270:1) node[draw] (6) { \tiny 6};
        \draw (315:1) node[draw] (7) { \tiny 7};
        \draw  (0) edge (1);
        \draw  (0) edge (5);
        \draw  (1) edge (2);
        \draw  (2) edge (3);
        \draw  (2) edge (4);
        \draw  (5) edge (6);
        \draw  (6) edge (7);
    \draw (0,-2) node () {$8_{1}$};
    \end{tikzpicture}
&
    \begin{tikzpicture}[scale=0.5,thick]
    \tikzstyle{every node}=[minimum width=0pt, inner sep=1pt, circle]
        \draw (0:1) node[draw] (0) { \tiny 0};
        \draw (45:1) node[draw] (1) { \tiny 1};
        \draw (90:1) node[draw] (2) { \tiny 2};
        \draw (135:1) node[draw] (3) { \tiny 3};
        \draw (180:1) node[draw] (4) { \tiny 4};
        \draw (225:1) node[draw] (5) { \tiny 5};
        \draw (270:1) node[draw] (6) { \tiny 6};
        \draw (315:1) node[draw] (7) { \tiny 7};
        \draw  (0) edge (1);
        \draw  (0) edge (5);
        \draw  (1) edge (2);
        \draw  (2) edge (3);
        \draw  (2) edge (4);
        \draw  (5) edge (6);
        \draw  (5) edge (7);
    \draw (0,-2) node () {$8_{2}$};
    \end{tikzpicture}
&
    \begin{tikzpicture}[scale=0.5,thick]
    \tikzstyle{every node}=[minimum width=0pt, inner sep=1pt, circle]
        \draw (0:1) node[draw] (0) { \tiny 0};
        \draw (45:1) node[draw] (1) { \tiny 1};
        \draw (90:1) node[draw] (2) { \tiny 2};
        \draw (135:1) node[draw] (3) { \tiny 3};
        \draw (180:1) node[draw] (4) { \tiny 4};
        \draw (225:1) node[draw] (5) { \tiny 5};
        \draw (270:1) node[draw] (6) { \tiny 6};
        \draw (315:1) node[draw] (7) { \tiny 7};
        \draw  (0) edge (1);
        \draw  (0) edge (5);
        \draw  (1) edge (2);
        \draw  (1) edge (4);
        \draw  (2) edge (3);
        \draw  (5) edge (6);
        \draw  (6) edge (7);
    \draw (0,-2) node () {$8_{3}$};
    \end{tikzpicture}
\\
    \begin{tikzpicture}[scale=0.5,thick]
    \tikzstyle{every node}=[minimum width=0pt, inner sep=1pt, circle]
        \draw (0:1) node[draw] (0) { \tiny 0};
        \draw (45:1) node[draw] (1) { \tiny 1};
        \draw (90:1) node[draw] (2) { \tiny 2};
        \draw (135:1) node[draw] (3) { \tiny 3};
        \draw (180:1) node[draw] (4) { \tiny 4};
        \draw (225:1) node[draw] (5) { \tiny 5};
        \draw (270:1) node[draw] (6) { \tiny 6};
        \draw (315:1) node[draw] (7) { \tiny 7};
        \draw  (0) edge (1);
        \draw  (0) edge (5);
        \draw  (1) edge (2);
        \draw  (1) edge (4);
        \draw  (2) edge (3);
        \draw  (5) edge (6);
        \draw  (5) edge (7);
    \draw (0,-2) node () {$8_{4}$};
    \end{tikzpicture}
&
    \begin{tikzpicture}[scale=0.5,thick]
    \tikzstyle{every node}=[minimum width=0pt, inner sep=1pt, circle]
        \draw (0:1) node[draw] (0) { \tiny 0};
        \draw (45:1) node[draw] (1) { \tiny 1};
        \draw (90:1) node[draw] (2) { \tiny 2};
        \draw (135:1) node[draw] (3) { \tiny 3};
        \draw (180:1) node[draw] (4) { \tiny 4};
        \draw (225:1) node[draw] (5) { \tiny 5};
        \draw (270:1) node[draw] (6) { \tiny 6};
        \draw (315:1) node[draw] (7) { \tiny 7};
        \draw  (0) edge (1);
        \draw  (0) edge (5);
        \draw  (0) edge (7);
        \draw  (1) edge (2);
        \draw  (1) edge (4);
        \draw  (2) edge (3);
        \draw  (5) edge (6);
    \draw (0,-2) node () {$8_{5}$};
    \end{tikzpicture}
&
    \begin{tikzpicture}[scale=0.5,thick]
    \tikzstyle{every node}=[minimum width=0pt, inner sep=1pt, circle]
        \draw (0:1) node[draw] (0) { \tiny 0};
        \draw (45:1) node[draw] (1) { \tiny 1};
        \draw (90:1) node[draw] (2) { \tiny 2};
        \draw (135:1) node[draw] (3) { \tiny 3};
        \draw (180:1) node[draw] (4) { \tiny 4};
        \draw (225:1) node[draw] (5) { \tiny 5};
        \draw (270:1) node[draw] (6) { \tiny 6};
        \draw (315:1) node[draw] (7) { \tiny 7};
        \draw  (0) edge (1);
        \draw  (0) edge (4);
        \draw  (0) edge (7);
        \draw  (1) edge (2);
        \draw  (2) edge (3);
        \draw  (4) edge (5);
        \draw  (5) edge (6);
    \draw (0,-2) node () {$8_{6}$};
    \end{tikzpicture}
&
    \begin{tikzpicture}[scale=0.5,thick]
    \tikzstyle{every node}=[minimum width=0pt, inner sep=1pt, circle]
        \draw (0:1) node[draw] (0) { \tiny 0};
        \draw (45:1) node[draw] (1) { \tiny 1};
        \draw (90:1) node[draw] (2) { \tiny 2};
        \draw (135:1) node[draw] (3) { \tiny 3};
        \draw (180:1) node[draw] (4) { \tiny 4};
        \draw (225:1) node[draw] (5) { \tiny 5};
        \draw (270:1) node[draw] (6) { \tiny 6};
        \draw (315:1) node[draw] (7) { \tiny 7};
        \draw  (0) edge (1);
        \draw  (0) edge (4);
        \draw  (1) edge (2);
        \draw  (2) edge (3);
        \draw  (4) edge (5);
        \draw  (4) edge (6);
        \draw  (4) edge (7);
    \draw (0,-2) node () {$8_{7}$};
    \end{tikzpicture}
&
    \begin{tikzpicture}[scale=0.5,thick]
    \tikzstyle{every node}=[minimum width=0pt, inner sep=1pt, circle]
        \draw (0:1) node[draw] (0) { \tiny 0};
        \draw (45:1) node[draw] (1) { \tiny 1};
        \draw (90:1) node[draw] (2) { \tiny 2};
        \draw (135:1) node[draw] (3) { \tiny 3};
        \draw (180:1) node[draw] (4) { \tiny 4};
        \draw (225:1) node[draw] (5) { \tiny 5};
        \draw (270:1) node[draw] (6) { \tiny 6};
        \draw (315:1) node[draw] (7) { \tiny 7};
        \draw  (0) edge (1);
        \draw  (0) edge (4);
        \draw  (0) edge (7);
        \draw  (1) edge (2);
        \draw  (2) edge (3);
        \draw  (4) edge (5);
        \draw  (4) edge (6);
    \draw (0,-2) node () {$8_{8}$};
    \end{tikzpicture}
&
    \begin{tikzpicture}[scale=0.5,thick]
    \tikzstyle{every node}=[minimum width=0pt, inner sep=1pt, circle]
        \draw (0:1) node[draw] (0) { \tiny 0};
        \draw (45:1) node[draw] (1) { \tiny 1};
        \draw (90:1) node[draw] (2) { \tiny 2};
        \draw (135:1) node[draw] (3) { \tiny 3};
        \draw (180:1) node[draw] (4) { \tiny 4};
        \draw (225:1) node[draw] (5) { \tiny 5};
        \draw (270:1) node[draw] (6) { \tiny 6};
        \draw (315:1) node[draw] (7) { \tiny 7};
        \draw  (0) edge (1);
        \draw  (0) edge (4);
        \draw  (0) edge (6);
        \draw  (1) edge (2);
        \draw  (2) edge (3);
        \draw  (4) edge (5);
        \draw  (6) edge (7);
    \draw (0,-2) node () {$8_{9}$};
    \end{tikzpicture}
&
    \begin{tikzpicture}[scale=0.5,thick]
    \tikzstyle{every node}=[minimum width=0pt, inner sep=1pt, circle]
        \draw (0:1) node[draw] (0) { \tiny 0};
        \draw (45:1) node[draw] (1) { \tiny 1};
        \draw (90:1) node[draw] (2) { \tiny 2};
        \draw (135:1) node[draw] (3) { \tiny 3};
        \draw (180:1) node[draw] (4) { \tiny 4};
        \draw (225:1) node[draw] (5) { \tiny 5};
        \draw (270:1) node[draw] (6) { \tiny 6};
        \draw (315:1) node[draw] (7) { \tiny 7};
        \draw  (0) edge (1);
        \draw  (0) edge (4);
        \draw  (0) edge (6);
        \draw  (0) edge (7);
        \draw  (1) edge (2);
        \draw  (2) edge (3);
        \draw  (4) edge (5);
    \draw (0,-2) node () {$8_{10}$};
    \end{tikzpicture}
\\
    \begin{tikzpicture}[scale=0.5,thick]
    \tikzstyle{every node}=[minimum width=0pt, inner sep=1pt, circle]
        \draw (0:1) node[draw] (0) { \tiny 0};
        \draw (45:1) node[draw] (1) { \tiny 1};
        \draw (90:1) node[draw] (2) { \tiny 2};
        \draw (135:1) node[draw] (3) { \tiny 3};
        \draw (180:1) node[draw] (4) { \tiny 4};
        \draw (225:1) node[draw] (5) { \tiny 5};
        \draw (270:1) node[draw] (6) { \tiny 6};
        \draw (315:1) node[draw] (7) { \tiny 7};
        \draw  (0) edge (1);
        \draw  (0) edge (6);
        \draw  (1) edge (2);
        \draw  (1) edge (3);
        \draw  (1) edge (4);
        \draw  (1) edge (5);
        \draw  (6) edge (7);
    \draw (0,-2) node () {$8_{11}$};
    \end{tikzpicture}
&
    \begin{tikzpicture}[scale=0.5,thick]
    \tikzstyle{every node}=[minimum width=0pt, inner sep=1pt, circle]
        \draw (0:1) node[draw] (0) { \tiny 0};
        \draw (45:1) node[draw] (1) { \tiny 1};
        \draw (90:1) node[draw] (2) { \tiny 2};
        \draw (135:1) node[draw] (3) { \tiny 3};
        \draw (180:1) node[draw] (4) { \tiny 4};
        \draw (225:1) node[draw] (5) { \tiny 5};
        \draw (270:1) node[draw] (6) { \tiny 6};
        \draw (315:1) node[draw] (7) { \tiny 7};
        \draw  (0) edge (1);
        \draw  (0) edge (5);
        \draw  (1) edge (2);
        \draw  (1) edge (3);
        \draw  (1) edge (4);
        \draw  (5) edge (6);
        \draw  (5) edge (7);
    \draw (0,-2) node () {$8_{12}$};
    \end{tikzpicture}
&
    \begin{tikzpicture}[scale=0.5,thick]
    \tikzstyle{every node}=[minimum width=0pt, inner sep=1pt, circle]
        \draw (0:1) node[draw] (0) { \tiny 0};
        \draw (45:1) node[draw] (1) { \tiny 1};
        \draw (90:1) node[draw] (2) { \tiny 2};
        \draw (135:1) node[draw] (3) { \tiny 3};
        \draw (180:1) node[draw] (4) { \tiny 4};
        \draw (225:1) node[draw] (5) { \tiny 5};
        \draw (270:1) node[draw] (6) { \tiny 6};
        \draw (315:1) node[draw] (7) { \tiny 7};
        \draw  (0) edge (1);
        \draw  (0) edge (5);
        \draw  (0) edge (7);
        \draw  (1) edge (2);
        \draw  (1) edge (3);
        \draw  (1) edge (4);
        \draw  (5) edge (6);
    \draw (0,-2) node () {$8_{13}$};
    \end{tikzpicture}
&
    \begin{tikzpicture}[scale=0.5,thick]
    \tikzstyle{every node}=[minimum width=0pt, inner sep=1pt, circle]
        \draw (0:1) node[draw] (0) { \tiny 0};
        \draw (45:1) node[draw] (1) { \tiny 1};
        \draw (90:1) node[draw] (2) { \tiny 2};
        \draw (135:1) node[draw] (3) { \tiny 3};
        \draw (180:1) node[draw] (4) { \tiny 4};
        \draw (225:1) node[draw] (5) { \tiny 5};
        \draw (270:1) node[draw] (6) { \tiny 6};
        \draw (315:1) node[draw] (7) { \tiny 7};
        \draw  (0) edge (1);
        \draw  (0) edge (5);
        \draw  (0) edge (6);
        \draw  (0) edge (7);
        \draw  (1) edge (2);
        \draw  (1) edge (3);
        \draw  (1) edge (4);
    \draw (0,-2) node () {$8_{14}$};
    \end{tikzpicture}
&
    \begin{tikzpicture}[scale=0.5,thick]
    \tikzstyle{every node}=[minimum width=0pt, inner sep=1pt, circle]
        \draw (0:1) node[draw] (0) { \tiny 0};
        \draw (45:1) node[draw] (1) { \tiny 1};
        \draw (90:1) node[draw] (2) { \tiny 2};
        \draw (135:1) node[draw] (3) { \tiny 3};
        \draw (180:1) node[draw] (4) { \tiny 4};
        \draw (225:1) node[draw] (5) { \tiny 5};
        \draw (270:1) node[draw] (6) { \tiny 6};
        \draw (315:1) node[draw] (7) { \tiny 7};
        \draw  (0) edge (1);
        \draw  (0) edge (4);
        \draw  (0) edge (7);
        \draw  (1) edge (2);
        \draw  (1) edge (3);
        \draw  (4) edge (5);
        \draw  (4) edge (6);
    \draw (0,-2) node () {$8_{15}$};
    \end{tikzpicture}
&
    \begin{tikzpicture}[scale=0.5,thick]
    \tikzstyle{every node}=[minimum width=0pt, inner sep=1pt, circle]
        \draw (0:1) node[draw] (0) { \tiny 0};
        \draw (45:1) node[draw] (1) { \tiny 1};
        \draw (90:1) node[draw] (2) { \tiny 2};
        \draw (135:1) node[draw] (3) { \tiny 3};
        \draw (180:1) node[draw] (4) { \tiny 4};
        \draw (225:1) node[draw] (5) { \tiny 5};
        \draw (270:1) node[draw] (6) { \tiny 6};
        \draw (315:1) node[draw] (7) { \tiny 7};
        \draw  (0) edge (1);
        \draw  (0) edge (4);
        \draw  (0) edge (6);
        \draw  (1) edge (2);
        \draw  (1) edge (3);
        \draw  (4) edge (5);
        \draw  (6) edge (7);
    \draw (0,-2) node () {$8_{16}$};
    \end{tikzpicture}
&
    \begin{tikzpicture}[scale=0.5,thick]
    \tikzstyle{every node}=[minimum width=0pt, inner sep=1pt, circle]
        \draw (0:1) node[draw] (0) { \tiny 0};
        \draw (45:1) node[draw] (1) { \tiny 1};
        \draw (90:1) node[draw] (2) { \tiny 2};
        \draw (135:1) node[draw] (3) { \tiny 3};
        \draw (180:1) node[draw] (4) { \tiny 4};
        \draw (225:1) node[draw] (5) { \tiny 5};
        \draw (270:1) node[draw] (6) { \tiny 6};
        \draw (315:1) node[draw] (7) { \tiny 7};
        \draw  (0) edge (1);
        \draw  (0) edge (4);
        \draw  (0) edge (6);
        \draw  (0) edge (7);
        \draw  (1) edge (2);
        \draw  (1) edge (3);
        \draw  (4) edge (5);
    \draw (0,-2) node () {$8_{17}$};
    \end{tikzpicture}
\\
&
    \begin{tikzpicture}[scale=0.5,thick]
    \tikzstyle{every node}=[minimum width=0pt, inner sep=1pt, circle]
        \draw (0:1) node[draw] (0) { \tiny 0};
        \draw (45:1) node[draw] (1) { \tiny 1};
        \draw (90:1) node[draw] (2) { \tiny 2};
        \draw (135:1) node[draw] (3) { \tiny 3};
        \draw (180:1) node[draw] (4) { \tiny 4};
        \draw (225:1) node[draw] (5) { \tiny 5};
        \draw (270:1) node[draw] (6) { \tiny 6};
        \draw (315:1) node[draw] (7) { \tiny 7};
        \draw  (0) edge (1);
        \draw  (0) edge (4);
        \draw  (0) edge (5);
        \draw  (0) edge (6);
        \draw  (0) edge (7);
        \draw  (1) edge (2);
        \draw  (1) edge (3);
    \draw (0,-2) node () {$8_{18}$};
    \end{tikzpicture}
&
    \begin{tikzpicture}[scale=0.5,thick]
    \tikzstyle{every node}=[minimum width=0pt, inner sep=1pt, circle]
        \draw (0:1) node[draw] (0) { \tiny 0};
        \draw (45:1) node[draw] (1) { \tiny 1};
        \draw (90:1) node[draw] (2) { \tiny 2};
        \draw (135:1) node[draw] (3) { \tiny 3};
        \draw (180:1) node[draw] (4) { \tiny 4};
        \draw (225:1) node[draw] (5) { \tiny 5};
        \draw (270:1) node[draw] (6) { \tiny 6};
        \draw (315:1) node[draw] (7) { \tiny 7};
        \draw  (0) edge (1);
        \draw  (0) edge (3);
        \draw  (0) edge (5);
        \draw  (0) edge (7);
        \draw  (1) edge (2);
        \draw  (3) edge (4);
        \draw  (5) edge (6);
    \draw (0,-2) node () {$8_{19}$};
    \end{tikzpicture}
&
    \begin{tikzpicture}[scale=0.5,thick]
    \tikzstyle{every node}=[minimum width=0pt, inner sep=1pt, circle]
        \draw (0:1) node[draw] (0) { \tiny 0};
        \draw (45:1) node[draw] (1) { \tiny 1};
        \draw (90:1) node[draw] (2) { \tiny 2};
        \draw (135:1) node[draw] (3) { \tiny 3};
        \draw (180:1) node[draw] (4) { \tiny 4};
        \draw (225:1) node[draw] (5) { \tiny 5};
        \draw (270:1) node[draw] (6) { \tiny 6};
        \draw (315:1) node[draw] (7) { \tiny 7};
        \draw  (0) edge (1);
        \draw  (0) edge (3);
        \draw  (0) edge (5);
        \draw  (0) edge (6);
        \draw  (0) edge (7);
        \draw  (1) edge (2);
        \draw  (3) edge (4);
    \draw (0,-2) node () {$8_{20}$};
    \end{tikzpicture}
&
    \begin{tikzpicture}[scale=0.5,thick]
    \tikzstyle{every node}=[minimum width=0pt, inner sep=1pt, circle]
        \draw (0:1) node[draw] (0) { \tiny 0};
        \draw (45:1) node[draw] (1) { \tiny 1};
        \draw (90:1) node[draw] (2) { \tiny 2};
        \draw (135:1) node[draw] (3) { \tiny 3};
        \draw (180:1) node[draw] (4) { \tiny 4};
        \draw (225:1) node[draw] (5) { \tiny 5};
        \draw (270:1) node[draw] (6) { \tiny 6};
        \draw (315:1) node[draw] (7) { \tiny 7};
        \draw  (0) edge (1);
        \draw  (0) edge (3);
        \draw  (0) edge (4);
        \draw  (0) edge (5);
        \draw  (0) edge (6);
        \draw  (0) edge (7);
        \draw  (1) edge (2);
    \draw (0,-2) node () {$8_{21}$};
    \end{tikzpicture}
&
    \begin{tikzpicture}[scale=0.5,thick]
    \tikzstyle{every node}=[minimum width=0pt, inner sep=1pt, circle]
        \draw (0:1) node[draw] (0) { \tiny 0};
        \draw (45:1) node[draw] (1) { \tiny 1};
        \draw (90:1) node[draw] (2) { \tiny 2};
        \draw (135:1) node[draw] (3) { \tiny 3};
        \draw (180:1) node[draw] (4) { \tiny 4};
        \draw (225:1) node[draw] (5) { \tiny 5};
        \draw (270:1) node[draw] (6) { \tiny 6};
        \draw (315:1) node[draw] (7) { \tiny 7};
        \draw  (0) edge (1);
        \draw  (0) edge (2);
        \draw  (0) edge (3);
        \draw  (0) edge (4);
        \draw  (0) edge (5);
        \draw  (0) edge (6);
        \draw  (0) edge (7);
    \draw (0,-2) node () {$8_{22}$};
    \end{tikzpicture}

    \end{tabular}
    \caption{Trees with at most 8 vertices. The indexing is used in Tables \ref{tab:identities1} and \ref{tab:identities2} to associate vertices with entries of the configurations.}
    \label{fig:Trees}
\end{figure}
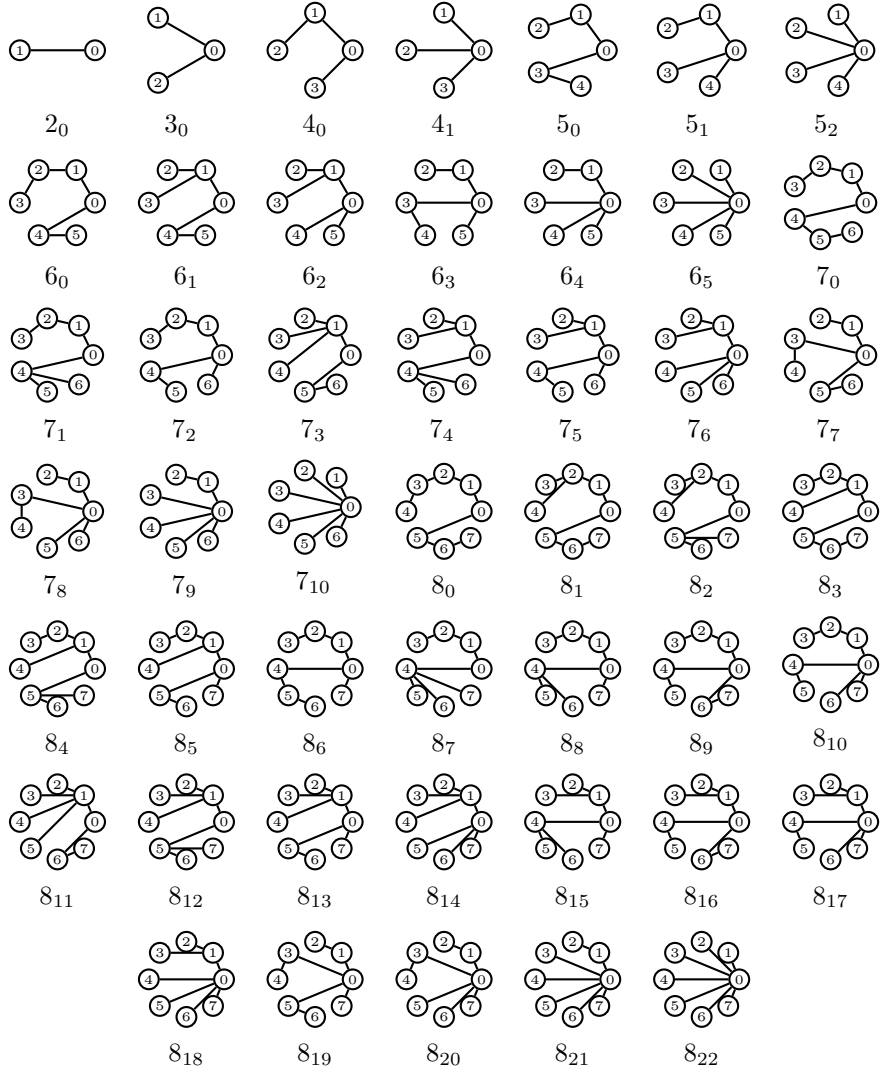

Among the recurrent configurations, the identity element is one of the most studied since it shows interesting patterns, see \cite[Section 5.7]{Klivans}.
In this section, we focus on the recurrent configurations associated with the identity element of the sandpile group of the dual graph of an outerplane graph where the vertex associated with the outer face is taken as sink.

Next result gives a method to compute the identity element.

\begin{proposition}\cite[Proposition 5.7.1]{Klivans}
Let $G$ be a connected graph with sink vertex $q$.
Let $\sigma_{max}\in\mathbb{N}^{V(G)}$ be the configuration in which the entry associated with vertex $v$ equals $\deg_G(v)-1$.
The recurrent configuration obtained from the stabilization
\[
s(2\sigma_{max}-s(2\sigma_{max}))
\]
is the identity element.
\end{proposition}

Given a tree $T$ with $n$ vertices and a vector $c\in\mathbb{N}^n$, whose entries are associated with the vertices of $T$, and $c$ is such that $c_v\geq \deg_T(v)$ for any non-leaf vertex $v\in T$ and $c_v\geq 2$ whenever $v$ is a leaf in $T$. 
Let $G_{T,c}$ be the planar graph obtained from $T$ by adding a sink vertex $q$ and adding $c_v-\deg_T(v)$ edges between the vertices $v$ and $q$, for each $v\in V(T)$.
Thus $c_v=\deg_{G_{T,c}}(v)$ for $v\in V(T)$.
The graphs $T$ and $G_{T,c}$ are the weak dual and dual of a family $F(T,c)$ of outerplane graphs.
Note the graphs in this family have the same sandpile group.

In figure~\ref{fig:Trees}, we give the trees with at most 8 vertices, the indexing on the vertices will be used to associate the entries of the configurations.

\begin{example}
Consider graph $6_2$ of Figure~\ref{fig:Trees}.
The graph $G(6_2,(4,5,3,3,3,3))$ is isomorphic to the dual graph of the plane graph $G$ of Figure~\ref{fig:exa2}.
Also, the graphs $G(6_2,(3,3,3,3,3,3))$ and $G(6_2,(6,5,3,3,3,3))$ are isomorphic to the dual graphs of the plane graphs $G_1$ and $G_2$ of Figure~\ref{fig:exa2.2}, respectively.
\end{example}

\begin{table}[ht]
    \centering
    \tiny
    \begin{tabular}{|c|cc|cc|}
    \hline
    $T$ & $c$ & identity & $c$ & identity \\
    \hline
$2_{0}$ & [2, 2] & [1, 1] & [3, 3] & [2, 2]\\
$3_{0}$ & [2, 2, 2] & [0, 1, 1] & [2, 3, 3] & [0, 2, 2]\\
$4_{0}$ & [2, 2, 2, 2] & [1, 1, 1, 1] & [2, 2, 3, 3] & [1, 1, 1, 1]\\
$4_{1}$ & [3, 2, 2, 2] & [0, 1, 1, 1] & [3, 3, 3, 3] & [0, 2, 2, 2]\\
$5_{0}$ & [2, 2, 2, 2, 2] & [0, 1, 1, 1, 1] & [2, 2, 3, 2, 3] & [0, 1, 1, 1, 1]\\
$5_{1}$ & [3, 2, 2, 2, 2] & [2, 1, 1, 1, 1] & [3, 2, 3, 3, 3] & [2, 1, 1, 1, 1]\\
$5_{2}$ & [4, 2, 2, 2, 2] & [0, 1, 1, 1, 1] & [4, 3, 3, 3, 3] & [0, 2, 2, 2, 2]\\
$6_{0}$ & [2, 2, 2, 2, 2, 2] & [1, 1, 1, 1, 1, 1] & [2, 2, 2, 3, 2, 3] & [1, 1, 1, 2, 1, 2]\\
$6_{1}$ & [2, 3, 2, 2, 2, 2] & [0, 2, 1, 1, 1, 1] & [2, 3, 3, 3, 2, 3] & [0, 2, 1, 1, 1, 1]\\
$6_{2}$ & [3, 3, 2, 2, 2, 2] & [2, 2, 1, 1, 1, 1] & [3, 3, 3, 3, 3, 3] & [2, 2, 1, 1, 1, 1]\\
$6_{3}$ & [3, 2, 2, 2, 2, 2] & [1, 1, 1, 1, 1, 1] & [3, 2, 3, 2, 3, 3] & [1, 1, 1, 1, 1, 1]\\
$6_{4}$ & [4, 2, 2, 2, 2, 2] & [3, 1, 1, 1, 1, 1] & [4, 2, 3, 3, 3, 3] & [3, 1, 1, 1, 1, 1]\\
$6_{5}$ & [5, 2, 2, 2, 2, 2] & [0, 1, 1, 1, 1, 1] & [5, 3, 3, 3, 3, 3] & [0, 2, 2, 2, 2, 2]\\
$7_{0}$ & [2, 2, 2, 2, 2, 2, 2] & [0, 1, 1, 1, 1, 1, 1] & [2, 2, 2, 3, 2, 2, 3] & [0, 1, 1, 2, 1, 1, 2]\\
$7_{1}$ & [2, 2, 2, 2, 3, 2, 2] & [1, 1, 0, 1, 1, 1, 1] & [2, 2, 2, 3, 3, 3, 3] & [1, 1, 0, 1, 1, 1, 1]\\
$7_{2}$ & [3, 2, 2, 2, 2, 2, 2] & [1, 0, 1, 1, 1, 1, 1] & [3, 2, 2, 3, 2, 3, 3] & [1, 0, 1, 1, 1, 1, 1]\\
$7_{3}$ & [2, 4, 2, 2, 2, 2, 2] & [0, 3, 1, 1, 1, 1, 1] & [2, 4, 3, 3, 3, 2, 3] & [0, 3, 1, 1, 1, 1, 1]\\
$7_{4}$ & [2, 3, 2, 2, 3, 2, 2] & [0, 2, 1, 1, 2, 1, 1] & [2, 3, 3, 3, 3, 3, 3] & [0, 2, 1, 1, 2, 1, 1]\\
$7_{5}$ & [3, 3, 2, 2, 2, 2, 2] & [1, 2, 1, 1, 1, 1, 1] & [3, 3, 3, 3, 2, 3, 3] & [1, 2, 1, 1, 1, 1, 1]\\
$7_{6}$ & [4, 3, 2, 2, 2, 2, 2] & [3, 2, 1, 1, 1, 1, 1] & [4, 3, 3, 3, 3, 3, 3] & [3, 2, 1, 1, 1, 1, 1]\\
$7_{7}$ & [3, 2, 2, 2, 2, 2, 2] & [0, 1, 1, 1, 1, 1, 1] & [3, 2, 3, 2, 3, 2, 3] & [0, 1, 1, 1, 1, 1, 1]\\
$7_{8}$ & [4, 2, 2, 2, 2, 2, 2] & [2, 1, 1, 1, 1, 1, 1] & [4, 2, 3, 2, 3, 3, 3] & [2, 1, 1, 1, 1, 1, 1]\\
$7_{9}$ & [5, 2, 2, 2, 2, 2, 2] & [4, 1, 1, 1, 1, 1, 1] & [5, 2, 3, 3, 3, 3, 3] & [4, 1, 1, 1, 1, 1, 1]\\
$7_{10}$ & [6, 2, 2, 2, 2, 2, 2] & [0, 1, 1, 1, 1, 1, 1] & [6, 3, 3, 3, 3, 3, 3] & [0, 2, 2, 2, 2, 2, 2]\\
$8_{0}$ & [2, 2, 2, 2, 2, 2, 2, 2] & [1, 1, 1, 1, 1, 1, 1, 1] & [2, 2, 2, 2, 3, 2, 2, 3] & [1, 1, 1, 1, 1, 1, 1, 1]\\
$8_{1}$ & [2, 2, 3, 2, 2, 2, 2, 2] & [1, 1, 0, 1, 1, 1, 1, 1] & [2, 2, 3, 3, 3, 2, 2, 3] & [1, 1, 0, 1, 1, 1, 1, 2]\\
$8_{2}$ & [2, 2, 3, 2, 2, 3, 2, 2] & [1, 1, 1, 1, 1, 1, 1, 1] & [2, 2, 3, 3, 3, 3, 3, 3] & [1, 1, 1, 1, 1, 1, 1, 1]\\
$8_{3}$ & [2, 3, 2, 2, 2, 2, 2, 2] & [1, 2, 1, 0, 1, 1, 1, 1] & [2, 3, 2, 3, 3, 2, 2, 3] & [1, 2, 1, 0, 2, 1, 1, 2]\\
$8_{4}$ & [2, 3, 2, 2, 2, 3, 2, 2] & [0, 1, 1, 1, 1, 2, 1, 1] & [2, 3, 2, 3, 3, 3, 3, 3] & [0, 1, 1, 1, 1, 2, 1, 1]\\
$8_{5}$ & [3, 3, 2, 2, 2, 2, 2, 2] & [1, 1, 1, 1, 1, 1, 1, 1] & [3, 3, 2, 3, 3, 2, 3, 3] & [1, 1, 1, 1, 1, 1, 1, 1]\\
$8_{6}$ & [3, 2, 2, 2, 2, 2, 2, 2] & [2, 1, 1, 0, 1, 1, 0, 1] & [3, 2, 2, 3, 2, 2, 3, 3] & [2, 1, 1, 0, 1, 1, 0, 2]\\
$8_{7}$ & [2, 2, 2, 2, 4, 2, 2, 2] & [1, 1, 0, 1, 2, 1, 1, 1] & [2, 2, 2, 3, 4, 3, 3, 3] & [1, 1, 0, 1, 2, 1, 1, 1]\\
$8_{8}$ & [3, 2, 2, 2, 3, 2, 2, 2] & [1, 0, 1, 1, 2, 1, 1, 1] & [3, 2, 2, 3, 3, 3, 3, 3] & [1, 0, 1, 1, 2, 1, 1, 1]\\
$8_{9}$ & [3, 2, 2, 2, 2, 2, 2, 2] & [2, 1, 1, 1, 1, 1, 1, 1] & [3, 2, 2, 3, 2, 3, 2, 3] & [2, 1, 1, 2, 1, 2, 1, 2]\\
$8_{10}$ & [4, 2, 2, 2, 2, 2, 2, 2] & [2, 0, 1, 1, 1, 1, 1, 1] & [4, 2, 2, 3, 2, 3, 3, 3] & [2, 0, 1, 1, 1, 1, 1, 1]\\
$8_{11}$ & [2, 5, 2, 2, 2, 2, 2, 2] & [0, 4, 1, 1, 1, 1, 1, 1] & [2, 5, 3, 3, 3, 3, 2, 3] & [0, 4, 1, 1, 1, 1, 1, 1]\\
$8_{12}$ & [2, 4, 2, 2, 2, 3, 2, 2] & [0, 3, 1, 1, 1, 2, 1, 1] & [2, 4, 3, 3, 3, 3, 3, 3] & [0, 3, 1, 1, 1, 2, 1, 1]\\
$8_{13}$ & [3, 4, 2, 2, 2, 2, 2, 2] & [1, 3, 1, 1, 1, 1, 1, 1] & [3, 4, 3, 3, 3, 2, 3, 3] & [1, 3, 1, 1, 1, 1, 1, 1]\\
$8_{14}$ & [4, 4, 2, 2, 2, 2, 2, 2] & [3, 3, 1, 1, 1, 1, 1, 1] & [4, 4, 3, 3, 3, 3, 3, 3] & [3, 3, 1, 1, 1, 1, 1, 1]\\
$8_{15}$ & [3, 3, 2, 2, 3, 2, 2, 2] & [1, 2, 1, 1, 2, 1, 1, 1] & [3, 3, 3, 3, 3, 3, 3, 3] & [1, 2, 1, 1, 2, 1, 1, 1]\\
$8_{16}$ & [3, 3, 2, 2, 2, 2, 2, 2] & [0, 2, 1, 1, 1, 1, 1, 1] & [3, 3, 3, 3, 2, 3, 2, 3] & [0, 2, 1, 1, 1, 1, 1, 1]\\
$8_{17}$ & [4, 3, 2, 2, 2, 2, 2, 2] & [2, 2, 1, 1, 1, 1, 1, 1] & [4, 3, 3, 3, 2, 3, 3, 3] & [2, 2, 1, 1, 1, 1, 1, 1]\\
$8_{18}$ & [5, 3, 2, 2, 2, 2, 2, 2] & [4, 2, 1, 1, 1, 1, 1, 1] & [5, 3, 3, 3, 3, 3, 3, 3] & [4, 2, 1, 1, 1, 1, 1, 1]\\
$8_{19}$ & [4, 2, 2, 2, 2, 2, 2, 2] & [1, 1, 1, 1, 1, 1, 1, 1] & [4, 2, 3, 2, 3, 2, 3, 3] & [1, 1, 1, 1, 1, 1, 1, 1]\\
$8_{20}$ & [5, 2, 2, 2, 2, 2, 2, 2] & [3, 1, 1, 1, 1, 1, 1, 1] & [5, 2, 3, 2, 3, 3, 3, 3] & [3, 1, 1, 1, 1, 1, 1, 1]\\
$8_{21}$ & [6, 2, 2, 2, 2, 2, 2, 2] & [5, 1, 1, 1, 1, 1, 1, 1] & [6, 2, 3, 3, 3, 3, 3, 3] & [5, 1, 1, 1, 1, 1, 1, 1]\\
$8_{22}$ & [7, 2, 2, 2, 2, 2, 2, 2] & [0, 1, 1, 1, 1, 1, 1, 1] & [7, 3, 3, 3, 3, 3, 3, 3] & [0, 2, 2, 2, 2, 2, 2, 2]\\
    \hline
    \end{tabular}
    \caption{The identity element of the sandpile group of $G_{T,c}$.}
    \label{tab:identities1}
\end{table}

\begin{table}[ht]
    \centering
    \tiny
    \begin{tabular}{|c|cc|cc|}
    \hline
    $T$ & $c$ & identity & $c$ & identity \\
    \hline
$2_{0}$ & [2, 2] & [1, 1] & [3, 3] & [2, 2]\\
$3_{0}$ & [3, 2, 2] & [1, 1, 1] & [4, 3, 3] & [2, 2, 2]\\
$4_{0}$ & [3, 3, 2, 2] & [1, 1, 1, 1] & [4, 4, 3, 3] & [2, 2, 2, 2]\\
$4_{1}$ & [4, 2, 2, 2] & [1, 1, 1, 1] & [5, 3, 3, 3] & [2, 2, 2, 2]\\
$5_{0}$ & [3, 3, 2, 3, 2] & [1, 1, 1, 1, 1] & [4, 4, 3, 4, 3] & [2, 2, 2, 2, 2]\\
$5_{1}$ & [4, 3, 2, 2, 2] & [1, 1, 1, 1, 1] & [5, 4, 3, 3, 3] & [2, 2, 2, 2, 2]\\
$5_{2}$ & [5, 2, 2, 2, 2] & [1, 1, 1, 1, 1] & [6, 3, 3, 3, 3] & [2, 2, 2, 2, 2]\\
$6_{0}$ & [3, 3, 3, 2, 3, 2] & [1, 1, 1, 1, 1, 1] & [4, 4, 4, 3, 4, 3] & [2, 2, 2, 2, 2, 2]\\
$6_{1}$ & [3, 4, 2, 2, 3, 2] & [1, 1, 1, 1, 1, 1] & [4, 5, 3, 3, 4, 3] & [2, 2, 2, 2, 2, 2]\\
$6_{2}$ & [4, 4, 2, 2, 2, 2] & [1, 1, 1, 1, 1, 1] & [5, 5, 3, 3, 3, 3] & [2, 2, 2, 2, 2, 2]\\
$6_{3}$ & [4, 3, 2, 3, 2, 2] & [1, 1, 1, 1, 1, 1] & [5, 4, 3, 4, 3, 3] & [2, 2, 2, 2, 2, 2]\\
$6_{4}$ & [5, 3, 2, 2, 2, 2] & [1, 1, 1, 1, 1, 1] & [6, 4, 3, 3, 3, 3] & [2, 2, 2, 2, 2, 2]\\
$6_{5}$ & [6, 2, 2, 2, 2, 2] & [1, 1, 1, 1, 1, 1] & [7, 3, 3, 3, 3, 3] & [2, 2, 2, 2, 2, 2]\\
$7_{0}$ & [3, 3, 3, 2, 3, 3, 2] & [1, 1, 1, 1, 1, 1, 1] & [4, 4, 4, 3, 4, 4, 3] & [2, 2, 2, 2, 2, 2, 2]\\
$7_{1}$ & [3, 3, 3, 2, 4, 2, 2] & [1, 1, 1, 1, 1, 1, 1] & [4, 4, 4, 3, 5, 3, 3] & [2, 2, 2, 2, 2, 2, 2]\\
$7_{2}$ & [4, 3, 3, 2, 3, 2, 2] & [1, 1, 1, 1, 1, 1, 1] & [5, 4, 4, 3, 4, 3, 3] & [2, 2, 2, 2, 2, 2, 2]\\
$7_{3}$ & [3, 5, 2, 2, 2, 3, 2] & [1, 1, 1, 1, 1, 1, 1] & [4, 6, 3, 3, 3, 4, 3] & [2, 2, 2, 2, 2, 2, 2]\\
$7_{4}$ & [3, 4, 2, 2, 4, 2, 2] & [1, 1, 1, 1, 1, 1, 1] & [4, 5, 3, 3, 5, 3, 3] & [2, 2, 2, 2, 2, 2, 2]\\
$7_{5}$ & [4, 4, 2, 2, 3, 2, 2] & [1, 1, 1, 1, 1, 1, 1] & [5, 5, 3, 3, 4, 3, 3] & [2, 2, 2, 2, 2, 2, 2]\\
$7_{6}$ & [5, 4, 2, 2, 2, 2, 2] & [1, 1, 1, 1, 1, 1, 1] & [6, 5, 3, 3, 3, 3, 3] & [2, 2, 2, 2, 2, 2, 2]\\
$7_{7}$ & [4, 3, 2, 3, 2, 3, 2] & [1, 1, 1, 1, 1, 1, 1] & [5, 4, 3, 4, 3, 4, 3] & [2, 2, 2, 2, 2, 2, 2]\\
$7_{8}$ & [5, 3, 2, 3, 2, 2, 2] & [1, 1, 1, 1, 1, 1, 1] & [6, 4, 3, 4, 3, 3, 3] & [2, 2, 2, 2, 2, 2, 2]\\
$7_{9}$ & [6, 3, 2, 2, 2, 2, 2] & [1, 1, 1, 1, 1, 1, 1] & [7, 4, 3, 3, 3, 3, 3] & [2, 2, 2, 2, 2, 2, 2]\\
$7_{10}$ & [7, 2, 2, 2, 2, 2, 2] & [1, 1, 1, 1, 1, 1, 1] & [8, 3, 3, 3, 3, 3, 3] & [2, 2, 2, 2, 2, 2, 2]\\
$8_{0}$ & [3, 3, 3, 3, 2, 3, 3, 2] & [1, 1, 1, 1, 1, 1, 1, 1] & [4, 4, 4, 4, 3, 4, 4, 3] & [2, 2, 2, 2, 2, 2, 2, 2]\\
$8_{1}$ & [3, 3, 4, 2, 2, 3, 3, 2] & [1, 1, 1, 1, 1, 1, 1, 1] & [4, 4, 5, 3, 3, 4, 4, 3] & [2, 2, 2, 2, 2, 2, 2, 2]\\
$8_{2}$ & [3, 3, 4, 2, 2, 4, 2, 2] & [1, 1, 1, 1, 1, 1, 1, 1] & [4, 4, 5, 3, 3, 5, 3, 3] & [2, 2, 2, 2, 2, 2, 2, 2]\\
$8_{3}$ & [3, 4, 3, 2, 2, 3, 3, 2] & [1, 1, 1, 1, 1, 1, 1, 1] & [4, 5, 4, 3, 3, 4, 4, 3] & [2, 2, 2, 2, 2, 2, 2, 2]\\
$8_{4}$ & [3, 4, 3, 2, 2, 4, 2, 2] & [1, 1, 1, 1, 1, 1, 1, 1] & [4, 5, 4, 3, 3, 5, 3, 3] & [2, 2, 2, 2, 2, 2, 2, 2]\\
$8_{5}$ & [4, 4, 3, 2, 2, 3, 2, 2] & [1, 1, 1, 1, 1, 1, 1, 1] & [5, 5, 4, 3, 3, 4, 3, 3] & [2, 2, 2, 2, 2, 2, 2, 2]\\
$8_{6}$ & [4, 3, 3, 2, 3, 3, 2, 2] & [1, 1, 1, 1, 1, 1, 1, 1] & [5, 4, 4, 3, 4, 4, 3, 3] & [2, 2, 2, 2, 2, 2, 2, 2]\\
$8_{7}$ & [3, 3, 3, 2, 5, 2, 2, 2] & [1, 1, 1, 1, 1, 1, 1, 1] & [4, 4, 4, 3, 6, 3, 3, 3] & [2, 2, 2, 2, 2, 2, 2, 2]\\
$8_{8}$ & [4, 3, 3, 2, 4, 2, 2, 2] & [1, 1, 1, 1, 1, 1, 1, 1] & [5, 4, 4, 3, 5, 3, 3, 3] & [2, 2, 2, 2, 2, 2, 2, 2]\\
$8_{9}$ & [4, 3, 3, 2, 3, 2, 3, 2] & [1, 1, 1, 1, 1, 1, 1, 1] & [5, 4, 4, 3, 4, 3, 4, 3] & [2, 2, 2, 2, 2, 2, 2, 2]\\
$8_{10}$ & [5, 3, 3, 2, 3, 2, 2, 2] & [1, 1, 1, 1, 1, 1, 1, 1] & [6, 4, 4, 3, 4, 3, 3, 3] & [2, 2, 2, 2, 2, 2, 2, 2]\\
$8_{11}$ & [3, 6, 2, 2, 2, 2, 3, 2] & [1, 1, 1, 1, 1, 1, 1, 1] & [4, 7, 3, 3, 3, 3, 4, 3] & [2, 2, 2, 2, 2, 2, 2, 2]\\
$8_{12}$ & [3, 5, 2, 2, 2, 4, 2, 2] & [1, 1, 1, 1, 1, 1, 1, 1] & [4, 6, 3, 3, 3, 5, 3, 3] & [2, 2, 2, 2, 2, 2, 2, 2]\\
$8_{13}$ & [4, 5, 2, 2, 2, 3, 2, 2] & [1, 1, 1, 1, 1, 1, 1, 1] & [5, 6, 3, 3, 3, 4, 3, 3] & [2, 2, 2, 2, 2, 2, 2, 2]\\
$8_{14}$ & [5, 5, 2, 2, 2, 2, 2, 2] & [1, 1, 1, 1, 1, 1, 1, 1] & [6, 6, 3, 3, 3, 3, 3, 3] & [2, 2, 2, 2, 2, 2, 2, 2]\\
$8_{15}$ & [4, 4, 2, 2, 4, 2, 2, 2] & [1, 1, 1, 1, 1, 1, 1, 1] & [5, 5, 3, 3, 5, 3, 3, 3] & [2, 2, 2, 2, 2, 2, 2, 2]\\
$8_{16}$ & [4, 4, 2, 2, 3, 2, 3, 2] & [1, 1, 1, 1, 1, 1, 1, 1] & [5, 5, 3, 3, 4, 3, 4, 3] & [2, 2, 2, 2, 2, 2, 2, 2]\\
$8_{17}$ & [5, 4, 2, 2, 3, 2, 2, 2] & [1, 1, 1, 1, 1, 1, 1, 1] & [6, 5, 3, 3, 4, 3, 3, 3] & [2, 2, 2, 2, 2, 2, 2, 2]\\
$8_{18}$ & [6, 4, 2, 2, 2, 2, 2, 2] & [1, 1, 1, 1, 1, 1, 1, 1] & [7, 5, 3, 3, 3, 3, 3, 3] & [2, 2, 2, 2, 2, 2, 2, 2]\\
$8_{19}$ & [5, 3, 2, 3, 2, 3, 2, 2] & [1, 1, 1, 1, 1, 1, 1, 1] & [6, 4, 3, 4, 3, 4, 3, 3] & [2, 2, 2, 2, 2, 2, 2, 2]\\
$8_{20}$ & [6, 3, 2, 3, 2, 2, 2, 2] & [1, 1, 1, 1, 1, 1, 1, 1] & [7, 4, 3, 4, 3, 3, 3, 3] & [2, 2, 2, 2, 2, 2, 2, 2]\\
$8_{21}$ & [7, 3, 2, 2, 2, 2, 2, 2] & [1, 1, 1, 1, 1, 1, 1, 1] & [8, 4, 3, 3, 3, 3, 3, 3] & [2, 2, 2, 2, 2, 2, 2, 2]\\
$8_{22}$ & [8, 2, 2, 2, 2, 2, 2, 2] & [1, 1, 1, 1, 1, 1, 1, 1] & [9, 3, 3, 3, 3, 3, 3, 3] & [2, 2, 2, 2, 2, 2, 2, 2]\\
    \hline
    \end{tabular}
    \caption{The identity element of the sandpile group of $G_{T,c}$.}
    \label{tab:identities2}
\end{table}

In Tables \ref{tab:identities1} and \ref{tab:identities2} is given the identity element of the sandpile group of $G_{T,c}$ for selected values of $d$. 
The entry of the sink has been omitted in the recurrent configurations.
For Table \ref{tab:identities1}, the graph $G_{T,c}$ obtained in the first and second column can be regarded as if $1$ and $2$ edges were added between each leaf of $T$ and the sink $q$, respectively.
For Table \ref{tab:identities2}, the graph $G_{T,c}$ obtained in the first and second column can be regarded as if $1$ and $2$ edges were added between each vertex of $T$ and the sink $q$, respectively.

There are many patterns in the identity element, for example, in Table~\ref{tab:identities1}, we see that the identity element of $G_{T,c}$ when $T$ is a star with at least 3 leaves and the leaves of $T$ are the only vertices connected with the sink, then the configuration 1 if the vertex is a leave and 0 otherwise is the identity element.
It is also interesting to see in Table~\ref{tab:identities2} that when the outerplane graph satisfy that exactly one edge of each inner face is adjacent with the outer face, then the identity element of the sandpile group of the dual with the outer face vertex as sink is the $\bf 1$ configuration.
An analogous result is observed when 2 faces are shared.
From which is conjectured that $G_{T,c}$ with $c=\deg(T)+k$, then the recurrent configuration is $k{\bf 1}$.

It is known that if $G$ is a planar graph and $G^*$ is a dual graph of $G$, then $K(G)\cong K(G^*)$.
And, there is an isomorphism between the recurrent configurations of $K(G)$ and the recurrent configurations of $K(G^*)$.
In \cite[Section 13.2]{perkinson}, a method was given to recover the recurrent configuration of a dual graph from a recurrent configuration of plane graph.
This method can be used to obtain the identity element of the sandpile group of the outerplane graphs whose dual is $G_{T,c}$.
In the following the method is described.

Let $H$ be a plane graph and $H^*$ be the dual graph.
Consider a planar drawing of $H$ and $H^*$ where each edge in $E(H)$ is crossed once by an edge in $E(H^*)$.
This associate bijectively the edges of $H$ with the edges of $H^*$.
An {\it orientation} of a graph is a choice of direction of each edge of the graph, and thus one end of the edge is the {\it head} and the other end is the {\it tail}.
Given an orientation of the edges of $H$, the {\it right-left rule} to orient the edges of $H^*$ consists in, for each edge $e\in E(H)$, following the direction of $e$, the direction of the associated edge $e^*\in E(H^*)$ goes from the right face to the left face separated by $e$.
Now, given a recurrent configuration $d$ of the sandpile group $K(H)$ with sink $q$, take  $d_q=-\sum_{v\in V(H)\setminus q}d_v$.
Consider an orientation of $H$, and orient the edges of $H^*$ following the right-left rule.
Find an $f\in {\mathbb Z}^{E(H)}$ such that $\partial(H)f=d$, where $\partial(H)$ is the {\it oriented incidence matrix}.
Take $f'\in{\mathbb Z}^{E(H^*)}$ such that $f'_{e^*}=f_{e}$.
The configuration $d'=\partial(H^*)f'$ is in the equivalence class of the recurrent configuration in $K(H^*)$ we are looking for.
To find the recurrent configuration in the class of $d'$, we suggest to use the following result.

\begin{proposition}\cite[Theorem 2.36]{Alfaro}\label{integer1}
Let $G$ be a graph with sink vertex $q$, and ${c}\in{\mathbb Z}^{V(G)\setminus q}$.
If ${x}^*$ is an optimal solution of the integer linear program 
\begin{eqnarray}\label{main:model}
\text{\rm maximize }   & & {\bf1}\cdot{x} \nonumber\\
\text{\rm subject to } & & {\bf 0}\leq { c}+{x}L_q(G) \leq \sigma_{max},\nonumber\\
	                       & & {x}\in {\mathbb Z}^{V(G)\setminus q}, \nonumber
\end{eqnarray}
then ${ x}^*$ is unique and ${ c} + { x}^* L_q(G)$ is a recurrent configuration in $SP(G,q)$ in the equivalence class of $c$.
\end{proposition}

\begin{figure}[ht]
    \centering
    \begin{tabular}{c@{\extracolsep{0cm}}c@{\extracolsep{0cm}}c}
         \begin{tikzpicture}[scale=.7,thick]
        	\tikzstyle{every node}=[minimum width=0pt, inner sep=1pt, circle]
        	
        	\clip (-3.5,-3.5) rectangle (2.5,3.5);
        	
            \draw (-1,0) + (120:2) node[draw,blue] (0) {\tiny 1};
            \draw (120:1) node[draw,blue] (1) {\tiny 2};
            \draw (60:2) node[draw,blue] (2) {\tiny 3};
            \draw (0:1) node[draw,blue] (3) {\tiny 4};
            \draw (-60:2) node[draw,blue] (4) {\tiny 5};
            \draw (240:1) node[draw,blue] (6) {\tiny 6};
            \draw (-1,0) + (240:2) node[draw,blue] (7) { \tiny $p$};
            \draw (180:2) node[draw,blue] (8) {\tiny 0};
            
            \draw  (1) edge[blue] (8);
        	\draw  (1) edge[blue] (6);
        	\draw  (6) edge[blue] (8);
        	\draw  (1) edge[blue] (3);
        	\draw  (3) edge[blue] (6);
        	
        	\draw[blue]  (0) -- (1) -- (2) -- (3) -- (4);
        	\draw[blue]  (4) -- (6) -- (7) -- (8) -- (0);
        	
            \draw (-1,0) + (120:1) node[draw] (11) {\tiny 4};
        	\draw (-1,0) + (240:1) node[draw] (12) {\tiny 5};
        	\draw (180:1) node[draw] (13) {\tiny 0};
        	\draw (0,0) node[draw] (14) {\tiny 1};
        	\draw (60:1) node[draw] (15) {\tiny 2};
        	\draw (-60:1) node[draw] (16) {\tiny 3};
        	\draw (-3,0) node[draw] (17) {\tiny $q$};
        	
        	\draw  (11) edge (13);
        	\draw  (13) edge (12);
        	\draw  (13) edge (14);
        	\draw  (14) edge (15);
        	\draw  (14) edge (16);
        	\draw (11) edge (17);
        	\draw (17) edge (12);
        	
        	\draw (17) to[out=90, looseness=3, in=80] (11);
        	\draw (17) to[out=90, looseness=2, in=120] (15);
        	\draw (17) to[out=100, looseness=3.4, in=20] (15);
        	\draw (17) to[out=-90, looseness=3, in=-80] (12);
        	\draw (17) to[out=-90, looseness=2, in=-120] (16);
        	\draw (17) to[out=-100, looseness=3.4, in=-20] (16);
        \end{tikzpicture}
         &
         \begin{tikzpicture}[scale=.7,thick]
        	\tikzstyle{every node}=[minimum width=0pt, inner sep=1pt, circle]
        	
        	\clip (-3.5,-3.5) rectangle (2.5,3.5);
        	
            
        	
        	
            \draw (-1,0) + (120:1) node[draw] (11) {\tiny 1};
        	\draw (-1,0) + (240:1) node[draw] (12) {\tiny 1};
        	\draw (180:1) node[draw] (13) {\tiny 2};
        	\draw (0,0) node[draw] (14) {\tiny 2};
        	\draw (60:1) node[draw] (15) {\tiny 1};
        	\draw (-60:1) node[draw] (16) {\tiny 1};
        	\draw (-3,0) node[draw] (17) {\tiny -8};
        	
        	\draw  (11) edge[<-] (13);
        	\draw  (13) edge[<-] (12);
        	\draw  (13) edge[<-] (14);
        	\draw  (14) edge[<-] (15);
        	\draw  (14) edge[<-] (16);
        	\draw[->] (11) edge (17);
        	\draw[<-] (17) edge (12);
        	
        	\draw[->] (17) to[out=90, looseness=3, in=80] (11);
        	\draw[->] (17) to[out=90, looseness=2, in=120] (15);
        	\draw[->] (17) to[out=100, looseness=3.4, in=20] (15);
        	\draw[->] (17) to[out=-90, looseness=3, in=-80] (12);
        	\draw[->] (17) to[out=-90, looseness=2, in=-120] (16);
        	\draw[->] (17) to[out=-100, looseness=3.4, in=-20] (16);
        	
        	\draw (-2.25,0.75) node[,red,] (21) {\tiny\bf -1};
        	\draw (-2.25,-0.75) node[,red,] (21) {\tiny\bf -1};
        	\draw (-1.45,0.25) node[,red,] (21) {\tiny\bf -1};
        	\draw (-1.1,-0.6) node[,red,] (21) {\tiny\bf 1};
        	\draw (-1.2,1.5) node[,red,] (21) {\tiny\bf 1};
        	\draw (-1.2,-1.5) node[,red,] (21) {\tiny\bf 1};
        	\draw (-0.4,-0.22) node[,red,] (21) {\tiny\bf 0};
        	\draw (-.18,1.5) node[,red,] (21) {\tiny\bf 1};
        	\draw (1.2,0.9) node[,red,] (21) {\tiny\bf 1};
        	\draw (-.18,-1.5) node[,red,] (21) {\tiny\bf 1};
        	\draw (1.2,-0.9) node[,red,] (21) {\tiny\bf 1};
        	\draw (0.48,0.4) node[,red,] (21) {\tiny\bf 1};
        	\draw (0.48,-0.4) node[,red,] (21) {\tiny\bf 1};
        \end{tikzpicture}
        &
         \begin{tikzpicture}[scale=.7,thick]
        	\tikzstyle{every node}=[minimum width=0pt, inner sep=1pt, circle]
        	
        	\clip (-2.5,-3.5) rectangle (1.5,3.5);
        	
            \draw (-1,0) + (120:2) node[draw,blue] (0) {\tiny 0};
            \draw (120:1) node[draw,blue] (1) {\tiny 0};
            \draw (60:2) node[draw,blue] (2) {\tiny 0};
            \draw (0:1) node[draw,blue] (3) {\tiny 0};
            \draw (-60:2) node[draw,blue] (4) {\tiny 0};
            \draw (240:1) node[draw,blue] (6) {\tiny 0};
            \draw (-1,0) + (240:2) node[draw,blue] (7) { \tiny 0};
            \draw (180:2) node[draw,blue] (8) {\tiny 0};
            
            \draw  (1) edge[blue,->] (8);
        	\draw  (1) edge[blue,->] (6);
        	\draw  (6) edge[blue,->] (8);
        	\draw  (1) edge[blue,->] (3);
        	\draw  (3) edge[blue,->] (6);
        	
        	\draw[blue] 
        	(0) edge[->] (1) 
        	(1) edge[->] (2) 
        	(2) edge[->] (3) 
        	(3) edge[->] (4)
        	(4) edge[->] (6)
        	(6) edge[->] (7)
        	(7) edge[<-] (8)
        	(8) edge[<-] (0);

        	\draw (-2.25,0.75) node[,red,] (21) {\tiny\bf -1};
        	\draw (-2.25,-0.75) node[,red,] (21) {\tiny\bf -1};
        	\draw (-1.2,0.25) node[,red,] (21) {\tiny\bf -1};
        	\draw (-1.3,-0.6) node[,red,] (21) {\tiny\bf 1};
        	\draw (-1.2,1.5) node[,red,] (21) {\tiny\bf 1};
        	\draw (-1.2,-1.5) node[,red,] (21) {\tiny\bf 1};
        	\draw (-0.33,-0.1) node[,red,] (21) {\tiny\bf 0};
        	\draw (0.0,1.4) node[,red,] (21) {\tiny\bf 1};
        	\draw (1.2,0.9) node[,red,] (21) {\tiny\bf 1};
        	\draw (0.1,-1.5) node[,red,] (21) {\tiny\bf 1};
        	\draw (1.2,-0.9) node[,red,] (21) {\tiny\bf 1};
        	\draw (0.4,0.6) node[,red,] (21) {\tiny\bf 1};
        	\draw (0.4,-0.6) node[,red,] (21) {\tiny\bf 1};
        \end{tikzpicture}
         \\
         (a) & (b) & (c)
    \end{tabular}
    \caption{Computation of a configuration in $H^*$ associated with the recurrent configuration in $H$. In (a) a drawing of a plane graph $H$ (black) and its dual $H^*$ (blue) is shown together with the indexing of the non-sink vertices. The vertices $q$ and $p$ are the sink vertices in $H$ and $H^*$, respectively. In (b) an element $f$ in ${\mathbb Z}^{E(H)}$ is shown colored in red such that $\partial(H)f=(2,2,1,1,1,1,-8)\in K(H)$. In (c) $f'$ is used to find a configuration in $K(H
   ^*)$.}
    \label{fig:computingconfiguration}
\end{figure}
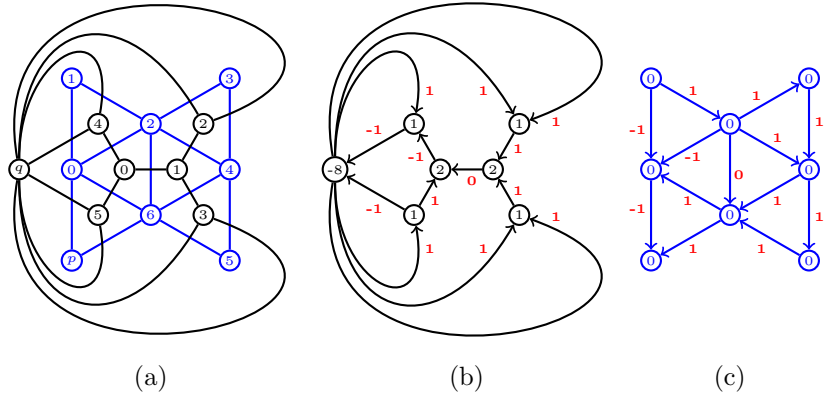

Let us see an example of the procedure to obtain a recurrent configuration in $K(H^*)$ given a configuration in $K(H)$.

\begin{example}
Let $H$ and $H^*$ be the black and blue plane graphs shown in Figure~\ref{fig:computingconfiguration}.a, where the sink vertices have index $q$ and $p$, respectively.
Note $H$ is isomorphic to the graph $G_{T,c}$ where $T$ is the tree $6_2$ in Figure~\ref{fig:Trees} and $c$ satisfy that the sink is adjacent only with the leaves by exactly 2 edges.
Following the indices described in Figure~\ref{fig:computingconfiguration}.a, the configuration $d=(2,2,1,1,1,1,-8)$ is the identity element of $K(H)$ up to the value of the sink $q$. Given the orientation of $H$ described in Figure~\ref{fig:computingconfiguration}.b, the oriented incidence matrix $\partial(H)$ of $H$ is the following:

\[
   \footnotesize\begin{array}{cc}
    & \begin{array}{ccccccccccccc}
    $04\,$ & $10\, $ & $21\,$ & $31\,$ & $4q\,$ & $50\,$ & $5q\,$ & $q2\,$ & $q2'$ & $q3\,$ & $q3'\,$& $q4\,$ & $q5$
    \end{array}\\
   \begin{array}{c}
   0\\
   1\\
   2\\
   3\\
   4\\
   5\\
   q
   \end{array} & \left(\begin{array}{cccccccccccccc}
    -1 &  1 &  0 &  0 &  0 &  1 &  0 &  0 &  0 &  0 &  0 &  0 &  0 \\
    0 & -1 &  1 &  1 &  0 &  0 &  0 &  0 &  0 &  0 &  0 &  0 &  0 \\
    0 &  0 & -1 &  0 &  0 &  0 &  0 &  1 &  1 &  0 &  0 &  0 &  0 \\
    0 &  0 &  0 & -1 &  0 &  0 &  0 &  0 &  0 &  1 &  1 &  0 &  0 \\
    1 &  0 &  0 &  0 & -1 &  0 &  0 &  0 &  0 &  0 &  0 &  1 &  0 \\
    0 &  0 &  0 &  0 &  0 & -1 & -1 &  0 &  0 &  0 &  0 &  0 &  1 \\
    0 &  0 &  0 &  0 &  1 &  0 &  1 & -1 & -1 & -1 & -1 & -1 & -1
  \end{array}\right).
  \end{array}
\]
Let $f=(-1,0,1,1,-1,1,-1,1,1,1,1,1,1)$. It can be seen that $f$ satisfy that $\partial(H)f=d$.
By using the right-left rule, we obtain the orientation of $H^*$ shown in Figure~\ref{fig:computingconfiguration}.c.
Thus, the oriented incidence matrix $\partial(H^*)$ is
\[
   \footnotesize\begin{array}{cc}
     & \begin{array}{cccccccccccccc}
    $\,0p\,$ & $\,10\,$ & $\,12\,$ & $\,20\,$ & $\,23\,$ & $\,24\,$ & $\,26\,$ & $34\,$ & $45\,$ & $\,46\,$ & $56\,$ & $\,60\,$ & $6p\,$
    \end{array}\\
    \begin{array}{c}
    0 \\
    1 \\
    2 \\
    3 \\
    4 \\
    5 \\
    p
    \end{array} & \left(\begin{array}{ccccccccccccc}
    -1 &  1 &  0 &  1 &  0 &  0 &  0 &  0 &  0 &  0 &  0 &  1 &  0 \\
    0 & -1 & -1 &  0 &  0 &  0 &  0 &  0 &  0 &  0 &  0 &  0 &  0 \\
    0 &  0 &  1 & -1 & -1 & -1 & -1 &  0 &  0 &  0 &  0 &  0 &  0 \\
    0 &  0 &  0 &  0 &  1 &  0 &  0 & -1 &  0 &  0 &  0 &  0 &  0 \\
    0 &  0 &  0 &  0 &  0 &  1 &  0 &  1 & -1 & -1 &  0 &  0 &  0 \\
    0 &  0 &  0 &  0 &  0 &  0 &  0 &  0 &  1 &  0 & -1 &  0 &  0 \\
    0 &  0 &  0 &  0 &  0 &  0 &  1 &  0 &  0 &  1 &  1 & -1 & -1 \\
    1 &  0 &  0 &  0 &  0 &  0 &  0 &  0 &  0 &  0 &  0 &  0 &  1 
  \end{array}\right).
  \end{array}
\]
Dualizing $f$, we get $f'=(-1,-1,1,-1,1,1,0,1,1,1,1,1,1)$.
From which we get the configuration $\partial(H^*)f'=(0,0,0,0,0,0,0,0)$.
Now, applying Proposition~\ref{integer1}, we get the following linear integer model:
\begin{eqnarray*}\label{main:model}
\text{\rm maximize }   & & \sum_{i=0}^6 x_i \nonumber\\
\text{\rm subject to } & & 0 \leq 4x_0-x_1-x_2-x_6 \leq 3\\
    & & 0 \leq -x_0+2x_1-x_2 \leq 1\\
    & & 0 \leq -x_0-x_1+5x_2-x_3-x_4-x_6 \leq 4 \\
    & & 0 \leq -x_2+2x_3-x_4 \leq 1 \\
    & & 0 \leq -x_2-x_3+4x_4-x_5-x_6 \leq 3 \\
    & & 0 \leq -x_4+2x_5-x_6 \leq 1 \\
    & & 0 \leq -x_0-x_2-x_4-x_5+5x_6 \leq 4 \\
	                   & & {x_i}\in {\mathbb Z} \text{ for each } i\in\{0,\dots,6\},
\end{eqnarray*}
whose optimal solution is $x^*=(5,6,7,7,7,7,6)$ and the recurrent configuration is $(1,0,4,0,1,1,4,p)$, which in fact is the identity element of the sandpile group of the outerplane graph $H^*$.
\end{example}

\section*{Acknowledgement}
This research was partially supported by SNI and CONACyT.
The authors are grateful to Prof. H.J. Fleischner for sending the authors some of his papers.

\end{document}